\documentclass[12pt]{article}
\usepackage{a4wide, amstext, amsmath, amssymb, graphicx, array, epsfig, psfrag}
\usepackage{amstext, amsmath, amssymb, graphicx, todonotes}
\usepackage{stmaryrd}
\usepackage[update,prepend]{epstopdf}

\newcommand{\bfI}{\boldsymbol I}

\newcommand{\bfV}{\boldsymbol V}

\newcommand{\mcK}{\mathcal{K}}
\newcommand{\mcE}{\mathcal{E}}

\newcommand{\mcF}{\mathcal{F}}

\newcommand{\mcT}{\mathcal{T}}

\newcommand{\mcH}{\mathcal{H}}

\newcommand{\mcv}{{H}}

\newcommand{\Sigmah}{{\Sigma_h}}
\newcommand{\nablas}{\nabla_\Sigma}
\newcommand{\nablash}{\nabla_{\Sigma_h}}

\newcommand{\lp}{\left(}
\newcommand{\rp}{\right)}

\newcommand{\tn}{|\mspace{-1mu}|\mspace{-1mu}|}

\numberwithin{equation}{section}
\newtheorem{lem}{Lemma}[section]

\newtheorem{thm}{Theorem}[section]
\newtheorem{rem}{Remark}[section]

\newenvironment{proof}{\noindent \newline {\bf Proof.}}
{\hfill \mbox{\fbox{} } \newline}
\begin{document}
\title{\bf Stabilized Finite Element Approximation
of the Mean Curvature Vector on Closed Surfaces
\thanks{This research was supported in part by the Swedish Foundation for Strategic Research Grant No.\ AM13-0029, and the Swedish Research Council Grants Nos.\ 2011-4992 and 2013-4708.}
}
 \author{
 Peter Hansbo\footnote{Department of Mechanical Engineering, J\"onk\"oping University, SE--551~11   J\"onk\"oping, Sweden, Peter.Hansbo@jth.hj.se } \mbox{ }
{Mats~G.~Larson} \footnote{Department of Mathematics and Mathematical Statistics, Ume{\aa} University, SE--901~87~~Ume{\aa}, Sweden, mats.larson@math.umu.se}
 \mbox{ }
{Sara~Zahedi} \footnote{Department of Mathematics, KTH, SE--100~44~~Stockholm, Sweden, sara.zahedi@math.kth.se}
}
\maketitle
\begin{abstract} We develop a stabilized discrete Laplace-Beltrami
operator that is used to compute an approximate mean curvature vector
which enjoys convergence of order one in $L^2$.
The stabilization is of gradient jump type and we consider both standard
meshed surfaces and so called cut surfaces that are level sets of
piecewise linear distance functions. We prove a priori error estimates and
verify the theoretical results numerically.
\end{abstract}
\section{Introduction} Accurate computation of the mean curvature vector
on a discrete surface plays an important role in computer graphics and
computational geometry, as well as in certain surface evolution problems,
see, e.g. \cite{BoKoPaAlLe10,BoSo08,CeHaLa14,DeMeScBa99,Dz91, Dz08}.

The mean curvature vector is obtained by letting the Laplace-Beltrami operator act on the embedding of the surface in ${\bf R}^3$ and various
formulas has been suggested in the literature, see \cite{MeyDesSchBa03}
and the references therein. It is known that the standard mean curvature vector based on the finite
element discrete Laplace-Beltrami operator for a piecewise linear triangulated smooth surface is of first order in $H^{-1}$, while no
order of convergence can, in general, be expected in $L^2$.
Convergence will also not occur in other standard methods, for
instance of finite difference type, without restrictive assumptions
on the mesh, see \cite{Xu13}. In \cite{GrRe07} estimates of order
$h^{1/2}$ in a $H^{-1/2}$ type norm, motivated by surface tension applications, is derived for an embedded interface defined by a
levelset function. Pointwise convergence results, without any factor
of the meshsize, was presented in \cite{HilPolWar06}.

In this paper we develop a stabilized version of the discrete Laplace-Beltrami operator. The stabilization consists of adding suitably scaled
gradient jumps to the $L^2$ projection involved in the definition of
the standard discrete Laplace-Beltrami operator. The stabilized method
produces a mean curvature vector that enjoys first order convergence in $L^2$. We consider two different types of piecewise linear approximations
of smooth surfaces. The first is the standard unstructured triangulation
and the second is a so called cut level set surface, which is the zero
level set of a piecewise linear continuous approximation of the distance function defined on a background mesh consisting of tetrahedra. In the
cut case an additional stabilization term on the faces of the background
mesh plays a crucial role. Such terms were originally proposed and
analyzed in \cite{BuHaLa13}. We prove a priori error estimates in the
$L^2$-norm for both cases and we also illustrate the results with
numerical examples.

The outline of the remainder of the paper is as follows:
In Section 2 we introduce the discrete surface approximations, in
Section 3 we define the stabilized mean curvature vector, in
Section 4 we develop the
theoretical framework and prove the a priori error estimate, and in
Section 5 we present numerical results confirming the theoretical
estimates.

\section{Meshed and Cut Discrete Surfaces}

\subsection{The Exact Surface}

Consider a closed smooth surface $\Sigma \subset {\bf R}^3$ with exterior
unit normal $n$. Let $\rho$ be the signed distance function such that
$\nabla \rho = n$ on $\Sigma$ and let $p(x) = x - \rho(x) n(p(x))$
be the closest point mapping. Let $U_\delta(\Gamma)$ be the open
tubular neighborhood 
$U_\delta(\Gamma) = \{x \in {\bf R}^3 : |\rho(x)|<\delta\}$ for $\delta>0$ of $\Sigma$. Then there is $\delta_0>0$
such that the closest point mapping $p(x)$ assigns precisely one point
on $\Sigma$ to each $x \in U_{\delta_0}(\Sigma)$. More precisely, we may  choose $\delta_0$ such that
\begin{equation}\label{eq:delta0}
\delta_0 \max(|\kappa_1(x)|,|\kappa_2(x)|) \leq C < 1 \quad \forall x \in \Sigma
\end{equation}
for some constant $C>0$. Here $\kappa_1(x)$ and $\kappa_2(x)$ are the principal curvatures at $x \in \Sigma$. See \cite{GiTr01}, Section 14.6
for further details.

\subsection{Approximation Properties}

We consider families of discrete connected piecewise linear surfaces
$\Sigma_h \subset U_{\delta_0}(\Sigma)$, where $0< h \leq h_0$ is a
mesh parameter and $h_0$ a small enough constant, that satisfy the
following approximation properties
\begin{align}\label{geomassum1}
\| \rho \|_{L^\infty(\Sigma_h)} & \lesssim h^2
\\ \label{geomassum2}
\| n\circ p  - n_h \|_{L^\infty(\Sigma_h)} &\lesssim h
\end{align}
Here and below we use the notation $\lesssim$ to denote less or
equal up to
a positive constant that is only dependent on given data and, in
particular, independent of the mesh parameter $h$.

We will consider two approaches to construct such piecewise linear
surfaces:
\begin{itemize}
\item Standard meshed surfaces where the surface consists of shape
regular triangles.
\item Cut surfaces that are piecewise planar iso--levels of a piecewise linear distance function defined on a background mesh consisting of
tetrahedra.
\end{itemize}
We shall treat meshed and cut surfaces in a unified
setting but certain concepts such as the mesh and later the
interpolation operator will be constructed in different ways.
However, the essential properties needed in the construction of the
Laplace-Beltrami operator and in the proof of the error estimate
are the same.

\subsection{Meshed Surface Approximation}

\begin{itemize}
\item Let $\Sigma_h = \cup_{K \in \mcK_h} K \subset U_{\delta_0}(\Sigma)$, be a quasiuniform triangulated surface with mesh parameter $0< h\leq h_0$, i.e.,
\begin{equation}
\text{Diam}(K)\lesssim h, \qquad \text{Diam}(K)/\text{diam}(K) \lesssim 1
\end{equation}
for all triangles $K$ in the mesh $\mcK_h$. Here $\text{Diam}(K)$ is the diameter
of $K$ and $\text{diam}(K)$ is the diameter of the largest inscribed circle in $K$.
\item Let $V_h$ be the space of piecewise linear continuous functions
defined on $\mcK_h$.
\end{itemize}

\subsection{Cut Surface Approximation}
\begin{itemize}
\item Let $\Omega_0$ be a polygon that contains $U_{\delta_0}(\Sigma)$.
Let $\mcT_{h,0}$ be a quasiuniform partition of $\Omega_0$
into shape regular tetrahedra $T$ with mesh parameter
$0< h \leq {h}_{\Omega_0}$, i.e.,
\begin{equation}
\text{Diam}(T)\lesssim h, \qquad \text{Diam}(T)/\text{diam}(T) \lesssim 1
\end{equation}
for all elements $T\in \mcT_{h,0}$. Let $\Sigmah \subset U_{\delta_0}(\Sigma)$
be a connected surface such that the intersection $\Sigmah \cap T$
is a subset of a hyperplane (or empty) for all $T\in \mcT_{h,0}$.
Let $\mcT_h = \{ T \in \mcT_{h,0} : T \cap \Sigmah \neq \emptyset\}$
and $\mcK_h = \{  \Sigmah \cap T : T \in \mcT_h \}$ and
let $h_0$, with $0< h_0\leq h_{\Omega_0}$, be chosen such that
$\cup_{T\in\mcT_h}T \subset U_{\delta_0}(\Sigma)$
for $0 < h \leq h_0$.
\item  Let $V_{h}$ be the space of piecewise linear continuous
functions on $\mcT_h$.
\end{itemize}
In practice, $\Sigmah$ is constructed by computing an approximation
$\rho_h$ of the levelset function $\rho$ associated with $\Sigma$
and then defining $\Sigmah$ as the zero levelset. Note that
$K \in \mcK_h$ will be a triangle or a planar quadrilateral.

\section{Stabilized Approximation of the Mean Curvature Vector}

\subsection{The Continuous Mean Curvature Vector}

The tangential gradient $\nabla_\Sigma$ is defined by
$\nabla_\Sigma = P_\Sigma \nabla$, where $\nabla$ is the
${\bf R}^3$ gradient and $P_\Sigma(x)
= I - n(x) \otimes n(x)$ is the projection onto the tangent plane $T_\Sigma(x)$ of $\Sigma$ at $x$.

The mean curvature vector $\mcv: \Sigma \rightarrow {\bf R}^3$
is defined by
\begin{equation}
\mcv = -\Delta_\Sigma x_\Sigma
\end{equation}
where $x_\Sigma: \Sigma \ni x \mapsto x \in {\bf R}^3$ is
the coordinate map or embedding of $\Sigma$ into ${\bf R}^3$ and
$\Delta_\Sigma = \nabla_\Sigma \cdot \nabla_\Sigma$ is the
Laplace-Beltrami operator. Note that for a general vector field 
$v:\Sigma \rightarrow {\bf R}^3$ the surface divergence 
$\text{div}_\Sigma v$ is defined by 
$\text{div}_\Sigma v=\text{tr}(v\otimes \nabla_\Sigma) 
=\text{tr}(v\otimes \nabla) - n \cdot (v \otimes \nabla )\cdot n$,  
and for tangent vector fields $v$ we have have the identity 
$\nabla_\Sigma \cdot v = \text{div}_\Sigma v$. 

The relation between the mean curvature vector and mean curvature is
given by the identity
\begin{equation}
\mcv = (\kappa_1 + \kappa_2)n
\end{equation}
where $\kappa_1$ and $\kappa_2$ are the two principal curvatures and
$(\kappa_1 +\kappa_2)/2$ is the mean curvature, see {\cite{BoKoPaAlLe10}}.

The mean curvature vector satisfies the following weak problem:
find $\mcv \in W=[H^1(\Sigma)]^3$ such that
\begin{equation}\label{meancurvature}
B(\mcv,v) = L(v) \quad \forall v \in W
\end{equation}
The forms are defined by
\begin{equation}
B(v,w) = (v, w)_\Sigma, \qquad L(w) = (\nablas x_{\Sigma}, \nablas w)_\Sigma
\end{equation}
where $\nablas w = w \otimes \nablas$ for a vector valued function $w$ and $(v,w)_\omega = \int_\omega v w dx$ is the $L^2$-inner product on the set $\omega$ with associated norm $\|v \|_\omega^2 = \int_\omega v^2 dx$.

We let $W^s_p(\omega)$ denote the standard Sobolev spaces on
$\omega \subseteq \Sigma$ or $\omega\subseteq {\bf R}^d$ with norm
$\|\cdot\|_{W^s_p(\omega)}$, see \cite{Wl87}. We also use the standard notation $W^s_p(\omega) = H^s(\omega)$ for $p=2$ and $W^0_p(\omega)
= L^p(\omega)$ for $s=0$. Since the surface is smooth we have the bound
\begin{equation}\label{regbound}
\| H \|_{W^s_\infty(\Sigma)} \lesssim 1
\end{equation}
for any choice of $s$ and we will, in particular, use this bound in our
analysis with $s=2$.

\subsection{The Stabilized Discrete Mean Curvature Vector}

Given the discrete coordinate map
$x_{\Sigma_h} : \Sigma_h \ni x \mapsto x \in {\bf R}^3$ we
define the stabilized discrete mean curvature vector
$\mcv_h$ as follows: find $\mcv_h \in W_h = [V_h]^3$ such that
\begin{equation}\label{meancurvaturedisc}
B_{h}(\mcv_h, v ) + J_h(\mcv_h,v)
= L_h ( v )
\quad \forall v \in W_h
\end{equation}
where the forms are defined by
\begin{align}
B_h(u,v) &= (u , v )_{\Sigma_h}
\\
L_h(v) &= (\nabla_{\Sigma_h} x_{\Sigma_h}, \nabla_{\Sigma_h} v )_{\Sigma_h}
\\
J_h(u,v) &=
\begin{cases}
\tau_{\mcE_h} J_{\mcE_h}(u,v)                   &\text{Meshed}
\\
\tau_{\mcE_h} J_{\mcE_h}(u,v)+\tau_{\mcF_h}J_{\mcF_h}(u,v)   &\text{Cut}
\end{cases}
\\
J_{\mcE_h}(u,v)&=\sum_{E \in \mcE_h}  h([t_E \cdot \nabla_{\Sigma_h} u],[t_E \cdot \nabla_{\Sigma_h} v])_E
\\
J_{\mcF_h}(u,v) &= \sum_{F \in \mcF_h} ([n_F \cdot \nabla u],[n_F \cdot \nabla v])_{F}
\end{align}
Here $\tau_{\mcE_h},\tau_{\mcF_h} \geq 0$ are parameters, $\mcE_h = \{E\}$ is the set of edges in the partition $\mcK_h$ of $\Sigma_h$, $\mcF_h=\{F\}$
is the set of interior faces in the partition $\mcT_{h}$. The jump
in the tangent gradient at an edge $E\in \mcE_h$ shared by
elements $K_1$ and $K_2$ in $\mcK_h$ is defined by
\begin{equation}
[t_E \cdot \nabla_{\Sigma_h} u] 
= t_{E,K_1} \cdot \nabla_{\Sigma_h} u_1 +  t_{E,K_2} \cdot \nabla_{\Sigma_h} u_2
\end{equation}
where $u_i = u|_{K_i}$, $i=1,2,$ and $t_{E,K_i}$ denotes the unit
vector orthogonal to $E$, tangent and exterior to $K_i$, $i=1,2.$ In the
same way the jump at a face $F \in \mcF_h$ shared by elements
$T_1$ and $T_2$ is defined by
\begin{equation}
[n_F \cdot \nabla u] = n_{F,T_1} \cdot \nabla u_1 +  n_{F,T_2} \cdot \nabla u_2
\end{equation}
where $n_{F,T_i}$ is the unit normal to the face $F$ exterior to element $T_i$, $i=1,2.$
\begin{rem}
The term $J_{\mcF_h}(u,v)$ is crucial in the cut case and enables us to essentially handle the cut case in the same way as the meshed case. 
It also stabilizes the possibly ill conditioned linear system of 
equations, see \cite{BuHaLa13}. In Theorem \ref{thm:jump} we will show 
that in the cut case it is indeed possible to take $\tau_{\mcE_h}=0$ 
and thus only add $J_{\mcF_h} (\cdot,\cdot)$. It is however convenient 
for the analysis to first include both the edge and face stabilization 
terms and then prove that only the face stabilization term is enough.
\end{rem}

\section{Error Estimates}

\subsection{Extension and Lifting of Functions}

\paragraph{Extension.}
Using the nearest point projection mapping any function $v$ on
$\Sigma$ can be extended to $U_{\delta_0}(\Sigma)$ using the
pull back
\begin{equation}
v^e = v \circ p \quad \text{on $U_{\delta_0}(\Sigma)$}
\end{equation}
Since the surface is smooth we have the stability estimate 
\begin{equation}\label{eq:extstab}
\|u^e\|_{W^s_\infty(U_{\delta_0}(\Sigma))}
\lesssim \|u\|_{W^s_\infty(\Sigma)}
\end{equation}
for $s>0$. We will, in particular, use $s=2$ in our forthcoming 
estimates. Using the chain rule we obtain
\begin{equation}\label{Dvefirst}
D v^e = D(v\circ p ) = Dv Dp = Dv (P_\Sigma - \rho \mcH)
\end{equation}
Here we used the identity $Dp = P_\Sigma - \rho \mcH$, where $\mcH$ is
the Hessian of the distance function
$\mcH= \nabla \otimes \nabla \rho$. For $x\in U_{\delta_0}(\Sigma)$ we
have
\begin{equation}\label{Hform}
\mcH(x) = \sum_{i=1}^2 \frac{\kappa_i^e}{1 + \rho(x)\kappa_i^e}a_i^e \otimes a_i^e
\end{equation}
where $\kappa_i$ are the principal curvatures with corresponding
principal curvature vectors $a_i$, see \cite{GiTr01} Lemma 14.7. 
Thus, using the bound (\ref{eq:delta0}) for $\delta_0$ we obtain
\begin{equation}\label{Hbound}
\|\mcH\|_{L^\infty(U_{\delta_0}(\Sigma))} \lesssim 1
\end{equation}
Starting from (\ref{Dvefirst}) we obtain
\begin{equation}
\nablash v^e = \nablash (v \circ p) 
= \nabla ( v \circ p) \cdot P_\Sigmah
= \nabla v \cdot D p P_\Sigmah 
= \nablas v \cdot P_\Sigma (P_\Sigma - \rho \mcH) P_\Sigmah
\end{equation}
where we used the fact that $P_\Sigma - \rho \mcH
= P_\Sigma(P_\Sigma - \rho \mcH)$, which follows from (\ref{Hform}). For each element $K \subset \Sigma_h $ and  $x \in K$ the resulting mapping
\begin{equation}\label{Bmap}
B = P_\Sigma (I - \rho \mcH)  P_\Sigmah : T_{x}(K)\rightarrow
T_{p(x)} (\Sigma)
\end{equation}
is invertible and we have the identity
\begin{equation}\label{Dve}
\nablash v^e = B^T \nablas v
\end{equation}

\paragraph{Lifting.}
The lifting $w^l$ of a function $w$ defined on $\Sigma_h$ to
$\Sigma$ is defined as the push forward
\begin{equation}
(w^l)^e = w^l \circ p = w \quad \text{on $\Sigma_h$}
\end{equation}
Using the chain rule we obtain
\begin{align}\label{Dwl}
D w &= D (w^l \circ p)  =  (D w^l)Dp
= (D w^l) (P_\Sigma - \rho \mcH)
\end{align}
and thus
\begin{align}
\nablash w &= \nabla (w^l \circ p) \cdot P_\Sigmah  =  \nabla (w^l) \cdot D p P_\Sigmah
\\
&\qquad = (\nabla w^l )\cdot (P_\Sigma - \rho \mcH) P_\Sigmah
= (\nablas w^l) \cdot P_\Sigma (P_\Sigma - \rho \mcH) P_\Sigmah
= (\nablas w^l) \cdot B
\end{align}
where $B$ is defined in (\ref{Bmap}). We obtain
\begin{equation}\label{Dvl}
\nablas w^l = B^{-T} \nablash w
\end{equation}

\paragraph{Estimates Related to $\boldsymbol{B}$.}
In order to prepare for the proof of the error estimate we collect some estimates
related to $B$. First
\begin{equation}\label{BBTbound}
 \| P_\Sigma - B B^T \|_{L^\infty(\Sigma)} \lesssim h^2,
 \quad \| B \|_{L^\infty(\Sigma_h)} \lesssim 1
 \quad \| B^{-1} \|_{L^\infty(\Sigma)} \lesssim 1
\end{equation}
Secondly we note that the surface measure $d \Sigma = |B| d \Sigmah$,
where $|B|$ is the absolute value of the determinant of $[B \xi_1 \, B \xi_2\, n^e]$ and $\{\xi_1,\xi_2\}$ is an orthonormal basis in $T_x(K)$,
and we have the following estimates
\begin{equation}\label{detBbound}
\| 1 - |B| \|_{L^\infty(\Sigma_h)} \lesssim h^2, \quad \||B|\|_{L^\infty(\Sigma_h)} \lesssim 1, \quad \||B|^{-1}\|_{L^\infty(\Sigma_h)} \lesssim 1
\end{equation}
see \cite{BuHaLa13} and \cite{De09}. In view of these bounds we note that
we have the following equivalences
\begin{equation}\label{eq:normequ}
\| v^l \|_{L^p(\Sigma)}
\sim \| v \|_{L^p(\Sigmah)}, \qquad \| v \|_{L^p(\Sigma)} \sim
\| v^e \|_{L^p(\Sigmah)}
\end{equation}
and
\begin{equation}\label{eq:normequgrad}
\| \nabla_\Sigma v^l \|_{L^p(\Sigma)} \sim \| \nablash v \|_{L^p(\Sigmah)},
\qquad \| \nablas v \|_{L^p(\Sigma)} \sim
\| \nablash v^e \|_{L^p(\Sigmah)}
\end{equation}

\subsection{Error Estimate for the Discrete Embedding}

Here we formulate an estimate of the difference between
the embeddings of the discrete and continuous surfaces.

\begin{lem}\label{lem:geomapprox} If the surface approximation
assumptions (\ref{geomassum1}) and (\ref{geomassum2}) hold, 
then
\begin{equation}
\| x_\Sigma^e - x_\Sigmah \|^2_{L^\infty(\Sigmah)}
+ h^2\| \nablash(x_\Sigma^e - x_\Sigmah) \|^2_{L^\infty(\Sigmah)}
\lesssim  h^4
\end{equation}
\end{lem}
\begin{proof} For the first term we have
\begin{equation}
\| x_\Sigma^e - x_\Sigmah \|_{L^\infty(\Sigmah)}
=\| \rho \|_{L^\infty(\Sigmah)} \lesssim h^2
\end{equation}
where we used (\ref{geomassum1}). For the second term
we have the identities
\begin{equation}
\nablash x_\Sigma^e = P_\Sigmah (P_\Sigma - \rho \mcH ),
\qquad \nablash x_{\Sigma_h}
= P_{\Sigma_h}
\end{equation}
and thus
\begin{align}
\| \nablash x_\Sigma^e - \nablash x_\Sigmah \|_{L^\infty(\Sigma_h)}
&\leq \| P_\Sigmah ( P_\Sigma - \rho \mcH ) - P_\Sigmah  \|_{L^\infty(\Sigma_h)}
\\
&\leq \| P_\Sigmah ( P_\Sigma  - P_\Sigmah ) \|_{L^\infty(\Sigma_h)}
+ \| \rho P_\Sigmah \mcH \|_{L^\infty(\Sigma_h)}
\\
&\lesssim h
\end{align}
where we used (\ref{geomassum1}), (\ref{geomassum2}), and (\ref{Hbound}).
%where we used the identity $(B - P_\Sigmah)P_\Sigmah
%= ((I-\rho\mcH)P_\Sigma  P_\Sigmah - P_\Sigmah)P_\Sigmah
%= (P_\Sigma - P_\Sigmah)P_\Sigmah - \rho \mcH P_\Sigma  P_\Sigmah$.
\end{proof}
%\todo[inline]{Fix this proof with more details.}

%
%
%\subsection{Error Estimate for the Discrete Embedding}
%
%Here we formulate an estimate of the error in the embeddings of
%the discrete and continuous surface surfaces.
%\begin{lem}\label{lem:geomapprox} If the surface approximation
%assumptions (\ref{geomassum1}) and (\ref{geomassum2}) hold. Then
%\begin{equation}
%\| x_\Sigma - x_\Sigmah^l \|^2_{\Sigma}
%+ h^2\| \nablas(x_\Sigma - x_\Sigmah^l) \|^2_{\Sigma}
%\lesssim  h^4
%\end{equation}
%\end{lem}
%\begin{proof}
%We note that the first term can be directly estimated in terms of the max
%norm. For the second we have the identities
%\begin{equation}
%\nablas x_\Sigma = P_\Sigma, \qquad \nablas x^l_{\Sigma_h}
%= B^{-T} P_{\Sigma_h}
%\end{equation}
%and thus
%\begin{align}
%\| \nablas x_\Sigma - \nablas x_\Sigmah^l \|_{\Sigma}
%&\leq \| P_\Sigma  - B^{-T} P_\Sigmah  \|_{\Sigma}
%\\
%%&\leq \| P_\Sigma - P_\Sigmah + B^{-T}( B - P_\Sigmah)P_\Sigmah \|_{\Sigma}
%%\\
%&\leq \| P_\Sigma - P_\Sigmah\|_{\Sigma} + \|B^{-T}( B - P_\Sigmah)P_\Sigmah \|_{\Sigma}
%\\
%&\leq \| P_\Sigma - P_\Sigmah\|_{\Sigma} + \|B^{-T}( P_\Sigma - P_\Sigmah -\rho\ct)P_\Sigmah \|_{\Sigma}
%\\
%&\leq \| P_\Sigma - P_\Sigmah\|_{\Sigma} + \|B^{-T} ( P_\Sigma - P_\Sigmah -\rho\ct)P_\Sigmah \|_{\Sigma}
%\\
%&\lesssim h
%\end{align}
%where we used the identity $(B - P_\Sigmah)P_\Sigmah
%= ((I-\rho\mcH)P_\Sigma  P_\Sigmah - P_\Sigmah)P_\Sigmah
%= (P_\Sigma - P_\Sigmah)P_\Sigmah - \rho \mcH P_\Sigma  P_\Sigmah$.
%\end{proof}
%%\todo[inline]{Fix this proof with more details.}

\subsection{Some Inequalities}
In this section we formulate some useful inequalities. First
a trace inequality that allows passage from an edge
$E \in \mcE_h$ to a tetrahedron $T\in \mcT_h$ for cut surfaces.
Then we prove two inverse inequalities.
For convenience we introduce the
semi norms
\begin{equation}
\tn v \tn_{\mcE_h}^2 = J_{\mcE_h}(v,v), \quad
\tn v \tn_{\mcF_h}^2 = J_{\mcF_h}(v,v)
\end{equation}

\begin{lem}\label{lem:doubletrace} In the cut case we have the 
following trace inequality
\begin{equation}\label{cuttrace}
\| v \|^2_E \lesssim h^{-2} \| v \|^2_T + \| \nabla v \|_T^2
+ h^{2}\| \nabla \otimes \nabla v \|_T^2\quad v\in H^2(T)
\end{equation}
where $E\in \mcE_h$, $T \in \mcT_h$, and $E \subset \partial T \cap \Sigma_h$.
\end{lem}
\begin{proof} We first apply the trace inequality
\begin{equation}
\| v \|^2_E \lesssim h^{-1} \| v \|^2_{F} + h \| \nabla_F v \|^2_F
\end{equation}
see Lemma 4.2 in \cite{HaHaLa03}, to pass from the edge $E$ to
the face $F = F(E)\in \mcF_h$ such that $E=F \cap \Sigma_h$. Then
we apply a standard trace inequality to pass from $F$ to an
element $T=T(F)\in \mcT_h$ to which $F$ is a face. More
precisely
\begin{align}
\| v \|_E^2 &\lesssim h^{-1} \| v \|_F^2 + h \| \nabla_F v \|^2_F
\\
&\lesssim  h^{-2} \| v \|_T^2 + \| \nabla v \|^2_T
+ \| \nabla_F v \|^2_T + h^2 \| (\nabla_F v)\otimes \nabla \|_T^2
\\
&\lesssim
h^{-2} \| v \|_T^2 + \| \nabla v \|^2_T
+ h^2 \| \nabla \otimes \nabla v \|_T^2
\end{align}
where $\nabla_F = P_F \nabla$, with $P_F = I - n_F \otimes n_F$
the constant projection onto the tangent plane of the face $F$,
is the tangent gradient to the face $K$. We also used the estimates
$\|\nabla_F v \|_T \lesssim \| \nabla v \|_T$ and $\|(\nabla_F v)\otimes \nabla \|_T =
\|(P_F \nabla v)\otimes \nabla \|_T
= \|P_F ((\nabla v)\otimes \nabla) \|_T \leq \| \nabla \otimes \nabla v \|_T^2$.
\end{proof}

\begin{lem}\label{lem:trace} The following inverse inequality holds
\begin{equation}
\sum_{E \in \mcE_h} h \| v \|^2_E \lesssim \| v \|^2_\Sigmah 
+ \tn v \tn^2_{\mcF_h} \quad \forall v \in V_h
\end{equation}
where $\tn v \tn^2_{\mcF_h}$ is present only in the cut case.
\end{lem}
\begin{proof} In the meshed case we have
\begin{equation}
\sum_{E \in \mcE_h} h \| v \|^2_E
\lesssim
\sum_{K \in \mcK_h} \| v \|^2_K + h^2 \| \nablash v \|^2_K
\lesssim
\sum_{K \in \mcK_h} \| v \|^2_K
=
\| v \|^2_{\Sigmah}
\end{equation}
where we used a standard trace inequality followed by an inverse
estimate. In the cut case we use Lemma \ref{lem:doubletrace} to get
\begin{align}
\sum_{E \in \mcE_h} h \| v \|^2_E
&\lesssim \sum_{T \in \mcT_h} h^{-1} \| v \|^2_{T}
+ h \| \nabla v \|^2_{T}
+ h^3 \| \nabla \otimes \nabla v \|^2_{T}
\\
&\lesssim \sum_{T \in \mcT_h} h^{-1} \| v \|^2_{T}
\\
&\lesssim \| v \|^2_{\Sigmah} + \tn v \tn^2_{\mcF_h}
\end{align}
where we used standard inverse inequalities and at last
Lemma 4.4 in \cite{BuHaLa13}.
%
%In the cut case we first use a trace inequality, Lemma 4.2 in
%\cite{HaHaLa03}, to pass from the edge $E\in \mcE_h$ to the
%corresponding face $F=F(E)\in \mcF_h$ such that $F \cap \Sigmah = E$
%and then we use a standard trace inequality to pass from the face
%$F$ to the element $T\in \mcT_h$. More precisely
%\begin{align}
%&\sum_{E \in \mcE_h} h \| v \|^2_E
%\lesssim \sum_{F\in \mcF_h} \| v \|^2_{F} + h^2 \| \nabla_F v\|^2_{F}
%\lesssim \sum_{F \in \mcF_h} \| v \|^2_{F}
%\\
%&\qquad
%\lesssim \sum_{T \in \mcT_h} h^{-1} \| v \|^2_{T} + h \| \nabla v \|^2_{T}
%\lesssim \sum_{T \in \mcT_h} h^{-1} \| v \|^2_{T}
%\lesssim \Big( \| v \|^2_{\Sigmah} + J_{\mcF_h}(v,v) \Big)
%\end{align}
%where we used standard inverse inequalities and at last Lemma 3.3 in \cite{BuHaLa13}. Here $\nabla_F$ denotes that tangent gradient at the
%face $F$.
\end{proof}
\begin{lem}\label{lem:grad} The following inverse inequality holds
\begin{equation}
h^2 \| \nablash v \|^2_\Sigmah \lesssim \| v \|^2_\Sigmah 
+ \tn v \tn^2_{\mcF_h} \quad \forall v \in V_h
\end{equation}
where $\tn v \tn^2_{\mcF_h}$ is present only in the cut case.
\end{lem}
\begin{proof} In the meshed case this estimate follows directly from a
standard elementwise inverse inequality. In the cut case we use the
fact that $(\nabla_\Sigmah v)|_K$ is constant
\begin{align}
h^2 \| \nablash v \|^2_\Sigmah &=\sum_{K\in\mcK_h} h^2 \| \nablash v \|_K^2
\\
&=\sum_{K \in\mcK_h} h^2 \text{meas}(K)(\text{meas}(T(K)))^{-1}
\| \nablash v \|_{T(K)}^2
\\
&\lesssim \sum_{T \in \mcT_h} h \|\nabla v \|^2_T
\\
&\lesssim \sum_{T \in \mcT_h} h^{-1} \| v \|^2_T
\\
&\lesssim \|v\|^2_{\Sigmah} + \tn v \tn^2_{\mcF_h} 
\end{align}
where $T(K)\in \mcT_h$ is the element such that $T \cap \Sigmah = K$ and we used standard inverse inequalities and at last Lemma 4.4 in \cite{BuHaLa13}.
\end{proof}

\subsection{Estimates for the Edge Stabilization Term}

In this section we prove two estimates for the edge stabilization 
term. The first shows that the edge stabilization term acting 
on an extension of a smooth function is $O(h^2)$. The second lemma 
is used in the proof of Theorem \ref{thm:jump} where we show that it 
is indeed enough to use the simplified stabilization $J_h(v,v) = \tau_{\mcF_h} J_{\mcF_h}(\cdot,\cdot)$ in the case of cut surfaces.

\begin{lem}\label{lem:tanjumpsmooth} If the surface approximation assumptions
(\ref{geomassum1}) and (\ref{geomassum2}) hold, then 
\begin{equation}\label{eq:tanjumpsmooth}
 \tn v^e \tn_{\mcE_h} 
 \lesssim h \| v \|_{W^2_\infty(\Sigma)}
\end{equation}
\end{lem}
\begin{proof} Consider the contribution $h \|[t_{E} \cdot \nablash v^e ]\|^2_E$ to $\tn v^e \tn^2_{\mcE_h}$ 
from edge $E\in \mcE_h$. Let $e_E$ be the unit vector parallel with the
edge $E$ such that $t_{E,K_1} = n_{h,1} \times e_E$ and $t_{E,K_2} = - n_{h,2} \times e_E$. Let $t=(n^e \times e_E)^l$ and $s = t \times n$.
Then $t$ and $s$ span the tangent plane $T_{p(x)}(\Sigma)$ for $x \in E$
and
\begin{align}
&\| t_{E,K_1} - t^e \|_{L^\infty(E)}
+ \| t^e + t_{E,K_2}\|_{L^\infty(E)}
\nonumber
\\ \label{tangenterrorest}
&\qquad = \|(n_{h,1} - n^e) \times e_E \|_{L^\infty(E)}
+ \|(n^e - n_{h,2}) \times e_E \|_{L^\infty(E)}
\lesssim h
\end{align}
We then have
\begin{align}
\| [t_{E} \cdot \nablash v^e ]\|_E
&=
\|(t_{E,K_1} + t_{E,K_2}) \cdot \nabla v^e \|_E
\\
&\leq \left( \|t_{E,K_1} - t^e\|_{L^\infty(E)}
        +  \|t^e +t_{E,K_2} \|_{L^\infty(E)} \right)
\| \nabla v^e\|_{E}
\\ \label{tanestb}
&\lesssim h \| \nabla v^e\|_{E}
\\
&\lesssim h^{3/2} \|v^e\|_{W^1_\infty(E)}
\\
&\lesssim h^{3/2} \| v \|_{W^1_\infty(\Sigma)}
\end{align}
where we used (\ref{tangenterrorest}) and the bound
$\|v^e\|_{W^1_\infty(E)} \lesssim
\| v^e \|_{W^1_\infty(U_{\delta_0}(\Sigma))}
\lesssim \| v \|_{W^1_\infty(\Sigma)}$, which 
follows from the stability (\ref{eq:extstab}) of extensions. 
Thus we obtain
\begin{equation}
\tn u^e \tn_{\mcE_h}^2 = \sum_{E \in \mcE_h} h \| [t_{E} \cdot \nablash u^e ]\|^2_E
\lesssim \sum_{E \in \mcE_h} h^4 \lesssim h^2
\end{equation}
since $\text{card}(\mcE_h) \lesssim h^{-2}$ both for meshed and
cut surfaces.
\end{proof}
\begin{lem}\label{lem:edgetermest} If the surface approximation
assumptions (\ref{geomassum1}) and (\ref{geomassum2}) hold, then 
the following bound holds for cut surfaces
\begin{equation}\label{eq:edgetermest}
\tn v \tn^2_{\mcE_h} \lesssim \| v \|^2_{\Sigma_h} + \tn v \tn^2_{\mcF_h}
\quad \forall v \in V_h
\end{equation}
\end{lem}
\begin{proof} Consider the contribution to $J_{\mcE_h}(v,v)$ from
an edge $E \in \mcE_h$. We employ the same notation as in the proof
of Lemma \ref{lem:tanjumpsmooth}.
%
%Let $e_E$ be the unit vector parallel with
%the edge $E$ such that $t_{E,K_1} = n_{h,1} \times e_E$
%and $t_{E,K_2} = - n_{h,2} \times e_E$. Let $t=(n^e \times e_E)^l$ and
%$s = t \times n$. Then $t$ and $s$ span the tangent plane
%$T_{p(x)}(\Sigma)$ for $x \in E$ and
%\begin{align}
%&\| t_{E,K_1} - t^e \|_{L^\infty(E)}
%+ \| t^e + t_{E,K_2}\|_{L^\infty(E)}
%\nonumber
%\\ \label{tangenterrorest}
%&\qquad = \|(n_{h,1} - n^e) \times e_E \|_{L^\infty(E)}
%+ \|(n^e - n_{h,2}) \times e_E \|_{L^\infty(E)}
%\lesssim h
%\end{align}
Adding and subtracting $t^e$, using some basic estimates,
the trace inequality in Lemma \ref{lem:doubletrace}, the
estimate (\ref{tangenterrorest}) for the tangent error,
and finally an inverse estimate give
\begin{align}\nonumber
&h\| t_{E,K_1}\cdot \nablash v_1 + t_{E,K_2}\cdot \nablash v_2\|^2_E
\\
%&\qquad =h\| t_{E,K_1}\cdot \nabla v_1 + t_{E,K_2} \cdot \nabla v_2\|^2_E
%\\
&\qquad \lesssim h\|t^e \cdot [\nabla v]\|^2_E
+ h \| (t_{E,K_1}-t^e) \cdot \nabla v_1\|^2_E
+ h \|(t^e+ t_{E,K_2}) \cdot \nabla v_2\|^2_E
\\
&\qquad \lesssim h \|t^e\cdot n_{F,T_1} [n_F \cdot \nabla v]\|_E^2
+ h^3 \|\nabla v_1 \|^2_E + h^3 \|\nabla v_2\|^2_E
\\
&\qquad \lesssim \|[n_F \cdot \nabla v]\|_F^2
+ h \|\nabla v_1 \|^2_{T_1} + h \|\nabla v_2\|^2_{T_2}
\\
&\qquad \lesssim \|[n_F \cdot \nabla v]\|_F^2
+ h^{-1} \| v_1 \|^2_{T_1} + h^{-1} \| v_2\|^2_{T_2}
\end{align}
Here $F$ is the face in $\mcF_h$ with $E=F\cap \Sigma_h$
and $T_1,T_2$ are the elements in $\mcT_h$ that share
the face $F$.
Using this estimate
we get
\begin{equation}
\tn v \tn^2_{\mcE_h} \lesssim \tn v \tn^2_{\mcF_h}
+ \sum_{T \in \mcT_h} h^{-1} \| v \|^2_T
\lesssim \| v \|^2_{\Sigmah} + \tn v \tn^2_{\mcF_h}
\end{equation}
where we used Lemma 4.4 in \cite{BuHaLa13} in the last estimate.
\end{proof}

\subsection{Stability Estimate for the Discrete Mean Curvature Vector}

In this section our main result is a stability estimate for the discrete mean curvature vector.

\begin{lem}\label{lem:stab} If the surface approximation assumptions
(\ref{geomassum1}) and (\ref{geomassum2}) hold and the stabilization 
parameters satisfy $0\leq \tau_{\mcE_h}$ and $0< \tau_{\mcF_h}$ (in 
the cut case), then the discrete mean curvature vector $\mcv_h$ 
defined by (\ref{meancurvaturedisc}) satisfies the stability
estimate
\begin{equation}\label{stability}
\|\mcv_h \|^2_{\Sigma_h} 
+ \tau_{\mcE_h} \tn \mcv_h \tn^2_{\mcE_h}
+ \tau_{\mcF_h} \tn \mcv_h \tn^2_{\mcF_h} 
\lesssim 1
\end{equation}
\end{lem}
\begin{rem} We note that in the meshed case the edge stabilization term 
is not necessary to prove $L^2(\Sigma_h)$ stability of $\mcv_h$ but 
with $\tau_{\mcE_h}>0$ we get stability in a stronger norm. However, in 
the cut case $\tau_{\mcF_h}$ must be strictly positive to establish the stability 
estimate. 
\end{rem}
\begin{proof} Setting $v = \mcv_h$ in (\ref{meancurvaturedisc}) we
obtain
\begin{align}\nonumber
&\|\mcv_h \|^2_{\Sigma_h} 
+ \tau_{\mcE_h} \tn \mcv_h \tn^2_{\mcE_h}
+ \tau_{\mcF_h} \tn \mcv_h \tn^2_{\mcF_h} 
\\
&\qquad=B_h(\mcv_h,\mcv_h) + J_h(\mcv_h,\mcv_h)
\\
&\qquad=
L_h(\mcv_h)
\\
&\qquad=(\nablash x_\Sigmah, \nablash \mcv_h )_{\Sigmah}
\\
&\qquad=(\nablash (x_\Sigmah - x_\Sigma^e), \nablash \mcv_h )_{\Sigmah}
+(\nablash x_\Sigma^e, \nablash \mcv_h )_{\Sigmah}
\\
&\qquad=I + II
\end{align}
\noindent{\bf Term $\bfI$.} Using the geometry approximation Lemma \ref{lem:geomapprox} followed by the inverse inequality in 
Lemma \ref{lem:grad} we obtain
\begin{align}\nonumber
|I|&=|(\nablash (x_\Sigmah - x^e_\Sigma ), \nablash \mcv_h )_{\Sigmah}|
\\
&
\lesssim \| \nablash (x_\Sigmah - x^e_\Sigma ) \|_\Sigmah
\|\nablash \mcv_h \|_\Sigmah
\\
&
\lesssim \delta^{-1} h^{-2} \| \nablash (x_\Sigmah - x^e_\Sigma ) \|^2_\Sigmah
+ \delta  h^2 \|\nablash \mcv_h \|^2_\Sigmah
\\ \label{stab:lemmac}
&
\lesssim \delta^{-1}
%\| \nablash (x_\Sigmah - x^e_\Sigma ) \|_\Sigmah
+ \delta \Big( \| \mcv_h \|^2_\Sigmah + \tn \mcv_h \tn^2_{\mcF_h} \Big)
%\\
%&\qquad
%\lesssim \delta^{-1} h^{-2} \| \nablash (x_\Sigmah - x^e_\Sigma ) \|^2_\Sigmah
%        +   \delta  \| \mcv_h \|^2_\Sigmah
%\\  \label{stab:lemmac}
%&\qquad
%\lesssim \delta^{-1}
%        +   \delta  \| \mcv_h \|^2_\Sigmah
\end{align}
for any $\delta>0$.
%\todo{Fix the inverse inequality here. Does not work in the cut case.}

\noindent{\bf Term $\bfI\bfI$.} Element wise partial integration gives
\begin{align}
|II|&=|(\nablash x^e_\Sigma, \nablash \mcv_h )_{\Sigmah}|
\\
&= \left| \sum_{E \in \mcE_h} ([t_E \cdot \nablash x^e_\Sigma],\mcv_h)_E \right|
\\
&\leq  \sum_{E \in \mcE_h} \|[t_E \cdot \nablash x^e_\Sigma]\|_E
\|\mcv_h\|_E
\\ \label{stab:lemmaa}
&\lesssim \delta^{-1} \sum_{E \in \mcE_h} h^{-1} \|[t_E \cdot \nablash x^e_\Sigma]\|^2_E
+
\delta \sum_{E \in \mcE_h}   h \|\mcv_h \|^2_E
\\ \label{stab:lemmab}
& \lesssim \delta^{-1}
+ \delta \Big( \| \mcv_h \|^2_{\Sigma_h} + \tn \mcv_h\tn^2_{\mcF_h} \Big)
\end{align}
Here the first term on the right hand side of (\ref{stab:lemmaa})
was estimated using Lemma \ref{lem:tanjumpsmooth} as follows
\begin{equation}
\sum_{E \in \mcE_h} h^{-1} \|[t_E \cdot \nablash x^e_\Sigma]\|^2_E
\lesssim h^{-2} \tn x^e_\Sigma \tn^2_{\mcE_h} \lesssim 1
\end{equation}
and the second term was estimated using Lemma \ref{lem:trace}.

Combining the bounds (\ref{stab:lemmac}) and (\ref{stab:lemmab})
of $I$ and $II$ we obtain
\begin{align}
\| \mcv_h \|^2_{\Sigma_h} 
+ \tau_{\mcE_h} \tn \mcv_h \tn^2_{\mcE_h}
+ \tau_{\mcF_h} \tn \mcv_h \tn^2_{\mcF_h}
&\lesssim \delta^{-1}
+ \delta \Big( \| \mcv_h \|^2_{\Sigma_h} + \tn \mcv_h \tn^2_{\mcF_h} \Big)
\end{align}
The desired bound is finally obtained by, using the fact that $\tau_{\mcF_h}>0$, in the cut case and choosing $\delta$ small
enough followed by a kick back argument.
\end{proof}

\subsection{Interpolation}
The construction of the interpolation operator is different in the meshed
and cut cases but we use the same notation for the operator to get a unified
treatment.

\paragraph{Meshed Case:}
Let $\pi_h:C(\Sigma) \rightarrow V_h$ be defined by
\begin{equation}
\pi_h: v \mapsto \pi_{L,\mcK_h} v^e
\end{equation}
where $\pi_{L,\mcK_h}$ is the Lagrange interpolation operator
defined on $\Sigmah$. We have the elementwise error estimate
\begin{equation}\label{interpoltriang}
\| v^e - \pi_h v \|_{H^{m}(K)} \lesssim h^{k-m} \| v^e \|_{H^k(K)}
, \quad 0\leq m \leq k \leq 2, \quad \forall K\in \mcK_h
\end{equation}

\paragraph{Cut Level Set Surface Case:}
Let $\pi_h:C(\Sigma) \rightarrow V_h$ be defined by
\begin{equation}
\pi_h: v \mapsto (\pi_{L,\mcT_h} v^e)
\end{equation}
where $\pi_{L,\mcT_h}$ is the Lagrange interpolation
operator defined on the three dimensional mesh $\mcT_h$.
We have the elementwise error estimates
\begin{equation}\label{interpoltet}
\| v^e - \pi_h v \|_{H^{m}(T)} \lesssim h^{k-m} \| v^e \|_{H^k(T)}
, \quad 0\leq m \leq k \leq 2, \quad \forall T\in \mcT_h
\end{equation}
and
\begin{equation}\label{interpolface}
\| v^e - \pi_h v \|_{H^{m}(F)} \lesssim h^{k-m} \| v^e \|_{H^k(F)}
, \quad 0\leq m \leq k \leq 2, \quad \forall F \in \mcF_h
\end{equation}

For convenience we shall use the simplified notation 
$\pi_h u = \pi_h u^e \in V_h$ and $\pi_h^l u = ((\pi_h u^e)_{\Sigmah})^l$. In both cases we have the following interpolation error estimate
\begin{align}\label{interpol}
\| u - \pi_h^l u \|_{H^m(\Sigma)} &\lesssim h^{k-m} \| u \|_{H^k(\Sigma)}, \quad 0\leq m \leq k \leq 2
\end{align}
See \cite{BuHaLa13} and \cite{De09} for a proof of (\ref{interpol}).
We will also need the following interpolation error estimate for the 
terms emanating from the stabilization. 

\begin{lem} If the surface approximation
assumptions (\ref{geomassum1}) and (\ref{geomassum2}) hold, 
then the following interpolation error estimates hold
\begin{align}
 \label{interpoledgenorm}
\tn u^e - \pi_h u^e \tn_{\mcE_h} &\lesssim h \| u \|_{W^2_\infty(\Sigma)}
\\ \label{interpolfacenorm}
\tn u^e - \pi_h u^e \tn_{\mcF_h} &\lesssim h \| u \|_{W^2_\infty(\Sigma)}
\end{align}
\end{lem}
\begin{proof}
\noindent \emph{Estimate (\ref{interpoledgenorm}).} In the meshed case applying a standard trace inequality elementwise followed by the interpolation estimate (\ref{interpoltriang}) yields
\begin{align}
&\sum_{E\in\mcE_h} h \| t_E \cdot \nablash (u^e - \pi_h u^e)\|^2_E
\nonumber
\\
&\qquad \lesssim
\sum_{K\in\mcK_h} \|\nablash (u^e - \pi_h u^e)\|^2_K
+ h^2 \|\nablash\otimes \nablash (u^e - \pi_h u^e)\|^2_K
\\
&\qquad \lesssim \sum_{K\in\mcK_h} h^2 \|u^e\|^2_{H^2(K)}
\\
&\qquad \lesssim \left( \sum_{K\in\mcK_h} h^4\right) \|u^e\|^2_{W^2_\infty(U_{\delta_0}(\Sigma))}
\\
&\qquad \lesssim h^2 \|u \|^2_{W^2_\infty(\Sigma)}
\end{align}
where we used the fact that $\text{card}(\mcK_h)\lesssim h^{-2}$
and the stability (\ref{eq:extstab}) of the extension $u^e$.

In the cut case, we first apply the trace inequality
(\ref{cuttrace}) to pass from the edge $E$ to the face
$F=F(E)$ such that $F\cap\Sigmah = E$. Next we note
that second order derivatives of $\pi_h u^e$ vanish,
then we use a trace inequality to pass from the faces to
the tetrahedra and use the interpolation estimate
(\ref{interpoltet}) as follows
\begin{align}
&\sum_{E \in \mcE_h} h \| t_E \cdot \nablash (u^e - \pi_h u^e)\|^2_E
\nonumber
\\
&\qquad \lesssim  \sum_{F \in \mcF_h} \|\nablash (u^e - \pi_h u^e) \|^2_F
+ h^2 \|\nabla_F (\nablash (u^e - \pi_h u^e))\|^2_F
\\
&\qquad \lesssim  \sum_{T \in \mcT_h}
h^{-1}\|\nablash (u^e - \pi_h u^e) \|^2_T + h \| \nabla(\nablash (u^e-\pi_h u^e)\|^2_T
\\ \nonumber
&\qquad\qquad + \sum_{F \in \mcF_h} h^2 \|\nabla(\nabla(u^e))\|^2_F
\\
&\qquad  \lesssim \sum_{T \in \mcT_h} h \|u^e\|_{H^2(T)}^2 +
\sum_{F \in \mcF_h} h^2 \|\nabla(\nabla(u^e))\|^2_F
\\
&\qquad  \lesssim \left(\sum_{T \in \mcT_h} h^4 +
\sum_{F \in \mcF_h} h^4 \right)
\| u^e \|^2_{W^2_\infty(U_{\delta_0}(\Sigma))}
\\ \label{interpolaa}
&\qquad  \lesssim h^2 \|u \|^2_{W^2_\infty(\Sigma)}
\end{align}
Here we used the fact that
$\text{card}(\mcF_h)\sim \text{card}(\mcT_h)\lesssim h^{-2}$
and the stability (\ref{eq:extstab}) of the extension $u^e$.
% the neighborhood $U_\delta(\Sigma)$ is chosen such that
%$\cup_{T\in \mcT_h} T \subset U_\delta(\Sigma)$ and thus we
%may take $\delta \sim h$, which leads to the estimate
%\begin{equation}\label{tubnbhdest}
% \| u^e \|^2_{H^2(U_{\delta}(\Sigma))}
% \lesssim  h \| u^e \|^2_{W^2_\infty(U_{\delta}(\Sigma))}
% \lesssim h \|u \|^2_{W^2_\infty(\Sigma)}
%\end{equation}

\noindent \emph{Estimate (\ref{interpolfacenorm}).} Using a
standard trace inequality followed by the interpolation estimate
(\ref{interpoltet}) we obtain
\begin{align}
\sum_{F \in \mcF_h} \|[ n_F \cdot \nabla (u^e - \pi_h u^e)]\|^2_F
&\lesssim \sum_{T \in \mcT_h} h^{-1} \|\nabla (u^e - \pi_h u^e )\|_T^2
+ h \|\nabla (\nabla (u^e - \pi_h u^e)) \|_T^2
\\
&\lesssim \sum_{T \in \mcT_h} h \|u^e \|_{H^2(T)}^2
\\
&\lesssim \left(\sum_{T \in \mcT_h} h^4\right) \|u^e \|_{W^2_\infty(U_{\delta_0}(\Sigma))}
\\
&\lesssim h^2 \| u \|^2_{W^2_\infty(\Sigma)}
\end{align}
where again we used the fact that $\text{card}(\mcT_h)\lesssim h^{-2}$
and the stability (\ref{eq:extstab}) of the extension $u^e$.
\end{proof}

\subsection{Error Estimate for the Discrete Mean Curvature Vector}

We are now ready to state and prove our main result.
\begin{thm}\label{thm:errorest} Let $\Sigma$ be a smooth surface,
$\Sigma_h$ an approximate
surface that is either meshed or cut and satisfies (\ref{geomassum1})
and (\ref{geomassum2}), then the discrete mean curvature vector $\mcv_h$,
defined by (\ref{meancurvaturedisc}), with parameters $\tau_{\mcE_h}>0$ 
and $\tau_{\mcF_h}>0$ (in the cut case), satisfies the estimate
\begin{equation}\label{l2bound}
\|\mcv - \mcv_h^l\|^2_\Sigma + \tau_{\mcE_h} \tn \mcv_h \tn^2_{\mcE_h}
+ \tau_{\mcF_h} \tn \mcv_h \tn^2_{\mcF_h} \lesssim h^2
\end{equation}
\end{thm}

\begin{proof} We first note that we have the following Galerkin
orthogonality property
\begin{align}
B(\mcv - \mcv_h^l, v^l) &= L(v^l) - B(\mcv_h^l, v^l)
\\
&=L(v^l) - L_h(v) + B_h(\mcv_h,v) - B(\mcv_h^l, v^l) + J_h(\mcv_h,v)
\end{align}
for all $v \in W_h$. Using this identity we obtain
\begin{align}
\nonumber
&B(\mcv - \mcv_h^l, \mcv - \mcv_h^l) + J_h(\mcv_h,\mcv_h)
\\
&\quad= B(\mcv - \mcv_h^l, \mcv - w^l)
+ B(\mcv - \mcv_h^l, w^l - \mcv_h^l) + J_h(\mcv_h,\mcv_h)
\\
&\quad=B(\mcv - \mcv_h^l, \mcv - w^l)
\\ \nonumber
& \qquad + L(w^l - \mcv_h^l) - L_h(w - \mcv_h)
\\ \nonumber
& \qquad
+ B_h(\mcv_h,w - \mcv_h)
- B(\mcv_h^l,w^l - \mcv_h^l)
\\  \nonumber
& \qquad
+ J_h(\mcv_h,w - \mcv_h ) + J_h(\mcv_h,\mcv_h)
\\
&=B(\mcv - \mcv_h^l, \mcv - w^l)
\\ \nonumber
& \qquad + \Big( L(w^l - \mcv) - L_h(w - \mcv^e) \Big)
+ \Big( L(\mcv - \mcv_h^l) - L_h(\mcv^e - \mcv_h)\Big)
\\ \nonumber
& \qquad
+ \Big( B_h(\mcv_h,w - \mcv_h)
- B(\mcv_h^l,w^l - \mcv_h^l) \Big)
+ J_h(\mcv_h,w)
\\
&=I + II + III + IV + V
\end{align}
for all $w \in W_h$. We choose $w = \pi_h \mcv$ and proceed
with estimates of terms $I-V$.

\paragraph{Term $\bfI$.} Using Cauchy-Schwarz followed by the
interpolation error estimate (\ref{interpol}), with $k=1$ and
$m=0$, we obtain
\begin{align}
|I| &= |B(\mcv - \mcv_h^l, \mcv - w^l)|
\\
&\leq \| \mcv - \mcv_h^l\|_{\Sigma} \|\mcv - w^l\|_{\Sigma}
\\
& \lesssim \delta^{-1} \|\mcv - w^l\|^2_{\Sigma}
+ \delta \| \mcv - \mcv_h^l\|^2_{\Sigma}
\\ \label{TermI}
&\lesssim \delta^{-1} h^2 + \delta \| \mcv - \mcv_h^l\|^2_{\Sigma}
\end{align}
for any $\delta>0$.

\paragraph{Term $\bfI\bfI$.} Changing domain of integration from
$\Sigma_h$ to $\Sigma$ and using Cauchy-Schwarz we obtain
\begin{align}
II &=L(w^l - \mcv) - L_h(w - \mcv^e)
\\
&= (\nablas x_\Sigma, \nablas (w^l - \mcv) )_\Sigma - (\nablash x_{\Sigma_h}, \nablash (w - \mcv^e) )_{\Sigma_h}
\\
%&= (\nablas x_\Sigma, \nablas (w^l - \mcv) )_\Sigma - (|B|^{-1}(\nablash %x_{\Sigma_h})^l, (\nablash (w - \mcv^e))^l )_{\Sigma}
%\\
&= (\nablas x_\Sigma, \nablas (w^l - \mcv) )_\Sigma
    - (|B|^{-1} B^T \nablas x_{\Sigma_h}^l, B^T \nablas (w^l - \mcv) )_{\Sigma}
\\
&= (\nablas x_\Sigma - |B|^{-1} B B^T \nablas x_{\Sigma_h}^l, \nablas (w^l - \mcv))_\Sigma
\\ \label{eq:termIIa}
&\leq \| \nablas x_\Sigma - |B|^{-1} B B^T \nablas x_{\Sigma_h}^l \|_\Sigma
\| \nablas (w^l - \mcv) \|_\Sigma
\\
&=II_1 \times II_2
\end{align}
\noindent{\emph{Term} $II_1$.} Adding and subtracting $x_\Sigmah^l$, using
the triangle inequality, and the equivalence of norms
(\ref{eq:normequgrad}), we obtain
\begin{align}\nonumber
II_1&=\| \nablas x_\Sigma - |B|^{-1} B B^T \nablas x_{\Sigma_h}^l \|_\Sigma
\\
&\leq \| \nablas (x_\Sigma - x_{\Sigma_h}^l)\|_\Sigma +
\|(P_\Sigma - |B|^{-1} B B^T)\nablas x_{\Sigma_h}^l \|_\Sigma
\\
&\lesssim \| \nablash (x_\Sigma^e - x_{\Sigma_h})\|_\Sigmah +
\|(P_\Sigma - |B|^{-1} B B^T)\|_{L^\infty(\Sigma)}
\|\nablas x_{\Sigma_h}^l \|_\Sigma
\\
&\lesssim h
\end{align}
Here we used the estimate
\begin{align}\nonumber
&\| P_\Sigma - |B|^{-1} B B^T \|_{L^\infty(\Sigma)}
=
\| |B|^{-1} (|B| P_\Sigma - B B^T )\|_{L^\infty(\Sigma)}
\\ \label{BBTtotbound}
&\qquad \leq \| |B|^{-1} (1 - |B|) \|_{L^\infty(\Sigma)}
+ \| |B|^{-1}\|_{L^\infty(\Sigma)} \| (P_\Sigma - B B^T )\|_{L^\infty(\Sigma)}
\lesssim h^2
\end{align}
which follows from the bounds (\ref{BBTbound}) and (\ref{detBbound})
for $B$ and its determinant. We also used the estimate
\begin{equation}
\|\nablas x_{\Sigma_h}^l \|_\Sigma \lesssim
\|\nablash (x_{\Sigma_h} - x_\Sigma^e) \|_\Sigmah + \|\nablas x_{\Sigma} \|_\Sigma
\lesssim h + 1 \lesssim 1
\end{equation}
where we used (\ref{eq:normequgrad}) and the first term was 
estimated using Lemma \ref{lem:geomapprox}.

\noindent{\emph{Term} $II_2$.} Using the interpolation error
estimate (\ref{interpol}), with $w^l = \pi_h^l \mcv$, we obtain
\begin{equation}
II_2 \lesssim h
\end{equation}
Combining the estimates of $II_1$ and $II_2$ we conclude that
\begin{equation}\label{TermII}
II \lesssim h^2
\end{equation}

\paragraph{Term $\bfI\bfI\bfI$.} Adding and subtracting a suitable
term yields
\begin{align}
III&=L(\mcv - \mcv_h^l) - L_h(\mcv^e - \mcv_h)
\\
&= \Big((\nablas x_\Sigma, \nablas (\mcv - \mcv_h^l) )_\Sigma
-(\nablash x_\Sigma^e, \nablash (\mcv^e - \mcv_h) )_{\Sigma_h}\Big)
\\ \nonumber
&\qquad + (\nablash(x_\Sigma^e - x_{\Sigma_h}),
\nablash (\mcv^e - \mcv_h) )_{\Sigma_h}
\\
&=III_1+III_2
\end{align}
We proceed with estimates of the terms $III_1$ and $III_2$.

\noindent \emph{Term $III_1$.} Changing domain of integration from $\Sigma_h$ to $\Sigma$ in the second term and using the bound
(\ref{BBTtotbound}) we get
\begin{align}
III_1&=(\nablas x_\Sigma, \nablas (\mcv - \mcv_h^l) )_\Sigma
-(|B|^{-1}(\nablash x_\Sigma^e)^l, (\nablash (\mcv^e - \mcv_h))^l )_{\Sigma}
\\
&=((P_\Sigma - |B|^{-1}BB^T)\nablas x_\Sigma, \nablas (\mcv - \mcv_h^l) )_\Sigma
\\
&\lesssim \| P_\Sigma -  |B|^{-1} B B^T \|_{L^\infty(\Sigma)}
\| \nablas x_\Sigma \|_\Sigma \|\nablas (\mcv - \mcv_h^l ) \|_\Sigma
\\
&\lesssim h^2 \|\nablas (\mcv - \mcv_h^l ) \|_\Sigma
\end{align}
Next continuing with
the estimate, we add and subtract an interpolant and use the
interpolation error estimate (\ref{interpol}) and the inverse
inequality in Lemma \ref{lem:grad} as follows
\begin{align}
&h^2 \|\nablas (\mcv - \mcv_h^l ) \|_\Sigma
\nonumber
\\
&\qquad \lesssim h^2 \|\nablas (\mcv - \pi_h^l \mcv ) \|_\Sigma
+
h^2 \|\nablas (\pi_h^l \mcv- \mcv_h^l ) \|_\Sigma
\\
&\qquad\lesssim h^3 + \delta^{-1} h^2
+  \delta h^2 \| \nablas(\pi_h^l \mcv - \mcv_h^l) \|^2_\Sigma
\\
&\qquad \lesssim h^3 + \delta^{-1} h^2
+ \delta  \Big( \|\pi_h^l \mcv - \mcv_h^l \|^2_\Sigma
+ \tn \pi_h \mcv^e -\mcv_h \tn_{\mcF_h}^2 \Big)
\\
&\qquad \lesssim h^3 + \delta^{-1} h^2 
\\ \nonumber
&\qquad \qquad + \delta
\Big(\|\mcv - \mcv_h^l\|^2_\Sigma + \|\mcv - \pi_h^l \mcv\|^2_\Sigma
+\tn \pi_h \mcv^e -\mcv^e \tn_{\mcF_h}^2 + \tn \mcv_h \tn_{\mcF_h}^2 \Big)
\\ \label{TermIII1}
&\qquad\lesssim h^3 + \delta^{-1} h^2 + \delta  h^2
+ \delta \Big( \|\mcv - \mcv_h^l \|^2_\Sigma
+  \tn \mcv_h \tn_{\mcF_h}^2  \Big)
\end{align}
where we used the interpolation error estimates (\ref{interpol}) 
and (\ref{interpolfacenorm}).

\noindent \emph{Term $III_2$.} Element wise partial integration gives
\begin{align}
III_2
&= \sum_{K \in \mcK_h} (\nablash (x_\Sigma^e - x_{\Sigma_h}),\nablash (\mcv^e - \mcv_h) )_{K}
\\
&=-\sum_{K \in \mcK_h} (x_\Sigma^e - x_{\Sigma_h},\Delta_{K} (\mcv^e - \mcv_h) )_{K}
\\ \nonumber
&\qquad + \sum_{E \in \mcE_h} (x_\Sigma^e - x_{\Sigma_h}, [t_{E} \cdot \nablash (\mcv^e - \mcv_h)] )_{E}
\\ \label{III2a}
&\lesssim \sum_{K \in \mcK_h} \|x_\Sigma^e - x_{\Sigma_h}\|_K
\| \Delta_{K} \mcv^e \|_{K}
\\ \nonumber
&\qquad + \delta^{-1}\left( \sum_{E \in \mcE_h} h^{-1} \| x_\Sigma^e - x_{\Sigma_h}\|^2_E \right)
+ \delta \tn \mcv^e - \mcv_h \tn_{\mcE_h}^2
% + \delta \left( \sum_{E \in \mcE_h} h \| [t_{E} \cdot \nablash (\mcv^e - %\mcv_h)]\|^2_E \right)
\end{align}
where $\Delta_K v = (\nablash \cdot \nablash v)|_K$ is the tangent
Laplacian on the flat element $K\in \mcK_h$ and therefore $\Delta_K H_h = 0$
since $H_h$ is linear on $K$. The first term on the right hand
side of (\ref{III2a}) is estimated using Lemma \ref{lem:geomapprox} as follows
\begin{equation}
 \sum_{K \in \mcK_h} \|x_\Sigma^e - x_{\Sigma_h}\|_K
\| \Delta_{K} \mcv^e \|_{K}
\lesssim \|x_\Sigma^e - x_{\Sigma_h}\|_\Sigmah
\|\mcv^e \|_{W^2_\infty(U_{\delta_0}(\Sigma))}
\lesssim h^2 \| H \|_{W^2_\infty(\Sigma)}
\lesssim h^2
\end{equation}
The second term is estimated
using Lemma \ref{lem:geomapprox} as follows
\begin{equation}
\sum_{E \in \mcE_h} h^{-1} \| x_\Sigma^e - x_{\Sigma_h} \|^2_E
\lesssim
\sum_{E \in \mcE_h} \| x_\Sigma^e - x_{\Sigma_h} \|^2_{L^\infty(E)}
\lesssim
\sum_{E \in \mcE_h} h^4 \lesssim h^2
\end{equation}
since card$(\mcE_h) \lesssim h^{-2}$ in both the meshed and cut case. Finally, the third term is estimated using Lemma \ref{lem:tanjumpsmooth}
as follows
\begin{equation}
\tn \mcv^e - \mcv_h \tn_{\mcE_h}^2
\lesssim 
\tn \mcv^e \tn_{\mcE_h}^2
+
\tn \mcv_h \tn_{\mcE_h}^2
\lesssim h^2 + \tn \mcv_h \tn^2_{\mcE_h}
\end{equation}
Thus we arrive at the bound
\begin{equation}\label{TermIII2}
III_2 \lesssim h^2 + \delta^{-1} h^2 + \delta h^2 +
\delta \tn \mcv_h \tn^2_{\mcE_h}
\end{equation}

Combining the estimates (\ref{TermIII1}) and (\ref{TermIII2}) of Terms
$III_1$ and $III_2$ we obtain
\begin{equation}\label{TermIII}
III \lesssim h^2 + \delta^{-1} h^2
+ \delta \left( \|\mcv - \mcv_h^l\|^2_\Sigma
+ \tn \mcv_h \tn_{\mcE_h}^2 
+ \tn \mcv_h \tn_{\mcF_h}^2
\right)
\end{equation}
for any $0< \delta\lesssim 1$.

\paragraph{Term $\bfI\bfV$.} Changing domain of integration from
$\Sigma$ to $\Sigma_h$ we obtain
\begin{align}
|IV| %&=  B_h(\mcv_h,w - \mcv_h) - B(\mcv_h^l,w^l - \mcv_h^l)
%\\
&= |(\mcv_h, w - \mcv_h )_\Sigmah - (\mcv_h^l, w^l - \mcv_h^l )_\Sigma|
\\
&= |(( 1 - |B|)\mcv_h, w - \mcv_h )_\Sigmah|
\\
&\lesssim h^2 \|\mcv_h \|_\Sigmah \| w - \mcv_h \|_\Sigmah
\\
&\lesssim h^2 \|\mcv_h \|_\Sigmah
\left( \| w \|_\Sigmah  + \| \mcv_h \|_\Sigmah \right)
\\ \label{TermIV}
&\lesssim h^2
\end{align}
where at last we used the following $L^\infty$ stability
of the interpolation operator
\begin{equation}
\|w\|_\Sigmah
= \|\pi_h H^e \|_\Sigmah
\leq \|\pi_h H^e \|_{L^\infty(\Sigmah)}
\leq \| H^e \|_{L^\infty(U_{\delta_0}(\Sigma))}
\lesssim \| H \|_{L^\infty(\Sigma)}
\lesssim 1
\end{equation}
which holds since $\pi_h$ is a Lagrange interpolation
operator and $\Sigmah \subset U_{\delta_0}(\Sigma)$
in the meshed case and $\cup_{T\in \mcT_h} T \subset
U_{\delta_0}(\Sigma)$ in the cut case, followed by
the stability  (\ref{stability}) of the discrete
curvature vector.

\paragraph{Term $\bfV$.} Adding and subtracting $\mcv^e$ inside 
the jump we obtain
\begin{align}
|V| &= |\tau_{\mcE_h} J_{\mcE_h}(H_h,w) + \tau_{\mcF_h} J_{\mcF_h}(H_h,w)|
\\
&=|\tau_{\mcE_h} J_{\mcE_h}(H_h,w-H^e)+\tau_{\mcE_h} J_{\mcE_h}(H_h,H^e)+\tau_{\mcF_h} J_{\mcF_h}(H_h,w-H^e)|
\\
&\lesssim \delta \left(\tau_{\mcE_h} \tn H_h \tn_{\mcE_h}^2
+ \tau_{\mcF_h} \tn H_h \tn_{\mcF_h}^2 \right)
\\ \nonumber
&\qquad + \delta^{-1} \left(\tau_{\mcE_h} \tn w-H^e \tn_{\mcE_h}^2 + \tau_{\mcE_h} \tn H^e \tn_{\mcE_h}^2
+ \tau_{\mcF_h}\tn w - H^e \tn^2_{\mcF_h} \right)
\\ \label{TermV}
&\lesssim \delta \left(\tau_{\mcE_h} \tn H_h \tn_{\mcE_h}^2
+ \tau_{\mcF_h} \tn H_h \tn_{\mcF_h}^2 \right) 
+ \delta^{-1} h^2
\end{align}
where we used the fact that $J_{\mcF_h}(H_h,H^e)=0,$ since
$[n_F \cdot \nabla H^e]=0$, the interpolation error estimates (\ref{interpoledgenorm}) and (\ref{interpolfacenorm}), and
Lemma \ref{lem:tanjumpsmooth} to estimate $\tn H^e \tn_{\mcE_h}$.

%In order to estimate the remaining term
%$\| H^e \|_{\mcE_h}$ we employ the same notation as in the proof of
%Lemma \ref{lem:edgetermest}. We then have
%\begin{align}
%\| [t_{E} \cdot \nablash \mcv^e ]\|_E
%&=
%\|(t_{E,K_1} + t_{E,K_2}) \cdot \nabla \mcv^e \|_E
%\\
%&\leq \left( \|t_{E,K_1} - t^e\|_{L^\infty(E)}
%        +  \|t^e +t_{E,K_2} \|_{L^\infty(E)} \right)
%\| \nabla \mcv^e\|_{E}
%\\ \label{tanestb}
%&\lesssim h \| \nabla \mcv^e\|_{E}
%\\
%&\lesssim h^{3/2}
%\end{align}
%where we used (\ref{tangenterrorest}) and at last the bound
%$\| \nabla \mcv^e\|_{E}\lesssim h^{1/2} \| \mcv \|_{W^1_\infty(\Sigma)}\lesssim h^{1/2}$. Thus we obtain
%\begin{equation}
%\tn H^e \tn_{\mcE_h}^2 = \sum_{E \in \mcE_h} h \| [t_{E} \cdot \nablash \mcv^e ]\|^2_E
%\lesssim \sum_{E \in \mcE_h} h^4 \lesssim h^2
%\end{equation}
%since $\text{card}(\mcE_h) \lesssim h^{-2}$ both for meshed and
%cut surfaces. We thus arrive at the estimate
%\begin{equation}\label{TermV}
%|V| \lesssim \delta^{-1} h^2 + \delta J_h(H_h,H_h)
%\end{equation}

\paragraph{Conclusion of the proof.} Collecting the estimates
(\ref{TermI}), (\ref{TermII}), (\ref{TermIII}), (\ref{TermIV}), and (\ref{TermV}), of terms $I-V$ we obtain
%\begin{align}\label{FINALLY}
%\|\mcv - \mcv_h^l\|^2_\Sigma
%+ J_h(\mcv_h,\mcv_h)
%&\lesssim h^2 + \delta^{-1} h^2
%\\ \nonumber
%&\qquad + \delta \left( \|\mcv - \mcv_h^l\|^2_\Sigma
%+ J_h(\mcv_h,\mcv_h)
%+ J_{\mcE_h}(\mcv_h,\mcv_h)
%+ J_{\mcF_h}(\mcv_h,\mcv_h)
%\right)
%\end{align}
\begin{align}\nonumber
&\|\mcv - \mcv_h^l\|^2_\Sigma + \tau_{\mcE_h} \tn H_h \tn_{\mcE_h}^2
+ \tau_{\mcE_h} \tn H_h \tn_{\mcF_h}^2
\\ \label{FINALLY}
&\qquad \lesssim h^2 + \delta^{-1} h^2
+ \delta \left( \|\mcv - \mcv_h^l\|^2_\Sigma
+ (1+\tau_{\mcE_h})\tn H_h \tn_{\mcE_h}^2 %J_{\mcE_h}(\mcv_h,\mcv_h)
+ (1+\tau_{\mcF_h})\tn H_h \tn_{\mcF_h}^2 %J_{\mcF_h}(\mcv_h,\mcv_h)
\right)
\end{align}
for any $0<\delta\lesssim 1$. Since $\tau_{\mcE_h}>0$
and $\tau_{\mcF_h}>0$ we may choose $\delta$ small enough
and conclude the proof using kick back argument.
\end{proof}
\begin{thm}\label{thm:jump} In the cut case we may take
to take $\tau_{\mcE_h}=0$ and thus use the simplified 
stabilization term
\begin{equation}
J_h(v,v) = \tau_{\mcF_h} J_{\mcF_h}(v,v)
\end{equation}
\end{thm}
\begin{proof} Using Lemma \ref{lem:edgetermest} and 
the interpolation error estimates (\ref{interpol}) and (\ref{interpoledgenorm}) we note that 
in the case of a cut surface we have the estimate
\begin{align}
\tn \mcv_h \tn^2_{\mcE_h} &\lesssim
\tn \mcv_h - \mcv^e \tn^2_{\mcE_h} + \tn \mcv^e \tn^2_{\mcE_h}
\\
&\lesssim \tn \mcv_h - \pi_h \mcv^e \tn^2_{\mcE_h}
+ \tn \pi_h \mcv^e - \mcv^e \tn^2_{\mcE_h}
+ \tn \mcv^e \tn^2_{\mcE_h}
\\
&\lesssim \|\mcv_h - \pi_h \mcv^e\|^2_{\Sigmah}
+ \tn \mcv_h - \pi_h \mcv^e \tn_{\mcF_h}^2 + h^2
\\
&\lesssim \|\mcv_h^l - \mcv\|^2_{\Sigma}
+ \tn \mcv_h \tn_{\mcF_h}^2 + h^2
\end{align}
where we finally used the interpolation estimates (\ref{interpol}) 
and (\ref{interpolfacenorm}). In view of the final estimate 
(\ref{FINALLY}) in the above proof, we conclude that in the 
cut case it is enough to use the simplified stabilization term
\begin{equation}
J_h(v,v) = \tau_{\mcF_h} J_{\mcF_h}(v,v)
\end{equation}
since the kick back term may be estimated as follows
\begin{align}\nonumber
&\|\mcv - \mcv_h^l\|^2_\Sigma
+ (1+\tau_{\mcE_h})\tn \mcv_h \tn_{\mcE_h}
+ (1+\tau_{\mcF_h})\tn \mcv_h \tn_{\mcF_h}
\\
&\qquad
\lesssim
\|\mcv - \mcv_h^l\|^2_\Sigma
+ (1+\tau_{\mcF_h})\tn \mcv_h\tn_{\mcF_h}
\end{align}
\end{proof}

\section{Numerical Examples}

\subsection{Triangulated Surfaces}

We consider a torus with Cartesian coordinates given by a map
from a reference coordinate system $(\theta,\varphi)$ representing
angles, $0\leq\varphi < 2\pi$, $0\leq\theta < 2\pi$:
\begin{equation}
\begin{cases}
x &  =(R+r\cos\varphi)\cos\theta\\
y & =(R+r\cos\varphi)\sin\theta\\
z & =r \sin\varphi
\end{cases}
\end{equation}
where $r$ is the radius of the tube bent into a torus and $R$ is the distance from the center line of the tube to the center of the torus.
The mean curvature is then given by
\[
H=-\frac{R+2 r \cos\varphi}{2 r (R+r \cos\varphi)}
\]
and we consider $R=1$, $r=1/2$, in our example.
%\[
%n = \left(2 x- 2 R x /\sqrt{x^2+y^2}, 2 y- 2 R y /\sqrt{x^2+y^2}, 2 z\right)
%\]

Our numerical results show that convergence of the mean curvature vector
is strongly dependent on stabilization. We compare three different meshes
on the torus, one sequence of structured meshes, Figure \ref{structmesh},
one where the diagonals are randomly flipped in the structured mesh,
Figure \ref{swapped}, and one where the nodes have been moved randomly, creating an unstructured mesh, Figure \ref{unstruct}.

In Figure \ref{convtorus} we show the discrete convergence $\|\pi_h \mcv- \mcv_h\|_{\Sigma_h}$, where $\pi_h \mcv$ is the nodal interpolant, for sequences of meshes of the type just described. The stabilization parameter was chosen as $\tau_{\mcE_h} = 1/10$ and the mesh size parameter $h={N}^{-1/2}$ where $N$ denotes the number of nodes in the mesh. We note that the structured mesh does not need stabilization whereas stability is lost even for the minor modification of flipping diagonals.
In Figure \ref{torus1}--\ref{torus2} we show iso--plots of the solution
for the structured mesh with flipped diagonals with and without stabilization. The instability of the computed curvature without added stabilization is clearly visible. We also note that the convergence rate is
higher than predicted by the theory. This may expected in view of the fact
that we have super convergence of second order on the structured mesh and
then loss of order is dependent on the perturbations of the mesh.

\subsection{Cut Level Set Surfaces}
We consider the same example as above. A structured mesh $\mcT_h$, consisting of tetrahedra, on the domain $[-1.6, 1.6] \times [-1.6, 1.6] \times [-0.6,0.6]$ is generated independently of the position of the
torus. The mesh size parameter is defined by $h=1/N^\frac{1}{3}$
where $N$ denotes the total number of nodes in the mesh. The signed
distance function of the torus $\Sigma$ is given by
\begin{equation}
\rho=\lp z^2+\lp (x^2+y^2)^{1/2}-R \rp^2\rp^{1/2}-r
\end{equation}
where we again choose $R=1$ and $r=1/2$. We construct an approximate
distance function $\rho_h$ using the nodal interpolant $\pi_h \rho$
on the background mesh and let $\Sigmah$ be the zero levelset of $\rho_h$.

We compare our approximation of the mean curvature vector with the exact mean curvature vector  $\mcv^e=-\lp \Delta\rho \rp \nabla \rho$.  Also in
this case the convergence of the mean curvature vector is strongly
dependent upon stabilization. In our example the stabilization parameters were chosen as $\tau_{\mcE_h} = 0$ and $\tau_{\mcF_h}=1/10$. Recall that
for a cut surface we may take $\tau_{\mcE_h}=0$, see Theorem \ref{thm:jump}. The resulting surface mesh $\mcK_h$ on $\Sigma_h$
is shown in Figure~\ref{meshtorus} and in Figure \ref{convcuttorus} we
show the error in the $L^2$-norm. We note that we also in this case
obtain higher order convergence rate (approximately 1.3) than predicted 
by the theory.

%\subsection{Cut Level set Surfaces}
%We consider the same example as above, i.e. we compute the mean curvature vector of a torus. The signed distance function $\rho=\sqrt{z^2+\lp \sqrt{x^2+y^2}-R \rp^2}-r$, where $R=1$, $r=1/2$. The computational domain is $[-1.6, 1.6] \times [-1.6, 1.6] \times [-0.6, 0.6]$.  We compare our approximation of the mean curvature vector with the solution $\mcv=\lp \nabla\cdot \nabla \rho\rp \nabla \rho$.  The convergence of the mean curvature vector is heavily dependent upon stabilization. The stabilization parameter was chosen as $\gamma = 10$ and the mesh size parameter $h=1/N^\frac{1}{3}$ where $N$ denotes the number of nodes in the mesh. A triangulation of $\Sigma_h$ is shown in Fig.~\ref{meshtorus}. In Fig.~\ref{convcuttorus} we show the error in the $L_2$-norm and the following mesh dependent norm
% \begin{align}
% %\|\bfkappa - \bfkappa_h^l\|_{\Sigma_h}  \nonumber \\
%\tn \mcv - \mcv_h \tn^2_{\Sigma_h}= \|\bfkappa - \bfkappa_h^l\|^2_{\Sigma_h} + \sum_{E\in \mcE} h \|[\bfn_E \cdot \nabla \bfkappa_h]\|^2_E
%\end{align}

\bibliographystyle{plain}
\bibliography{MeanCurv}

\begin{figure}[h]
\begin{center}
\includegraphics[height=9cm]{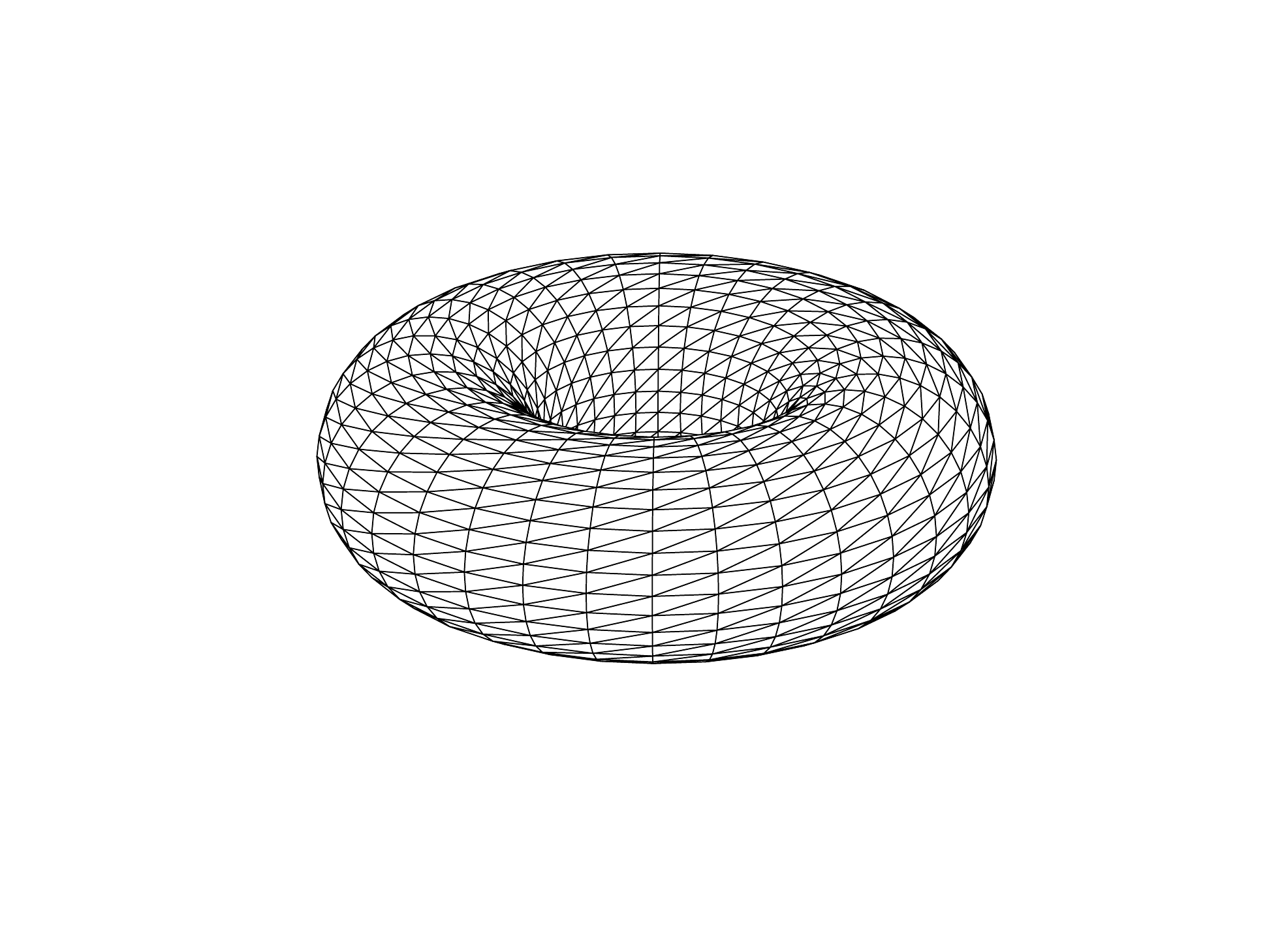}
\end{center}
\caption{Structured mesh.\label{structmesh}}
\end{figure}
\begin{figure}[h]
\begin{center}
\includegraphics[height=9cm]{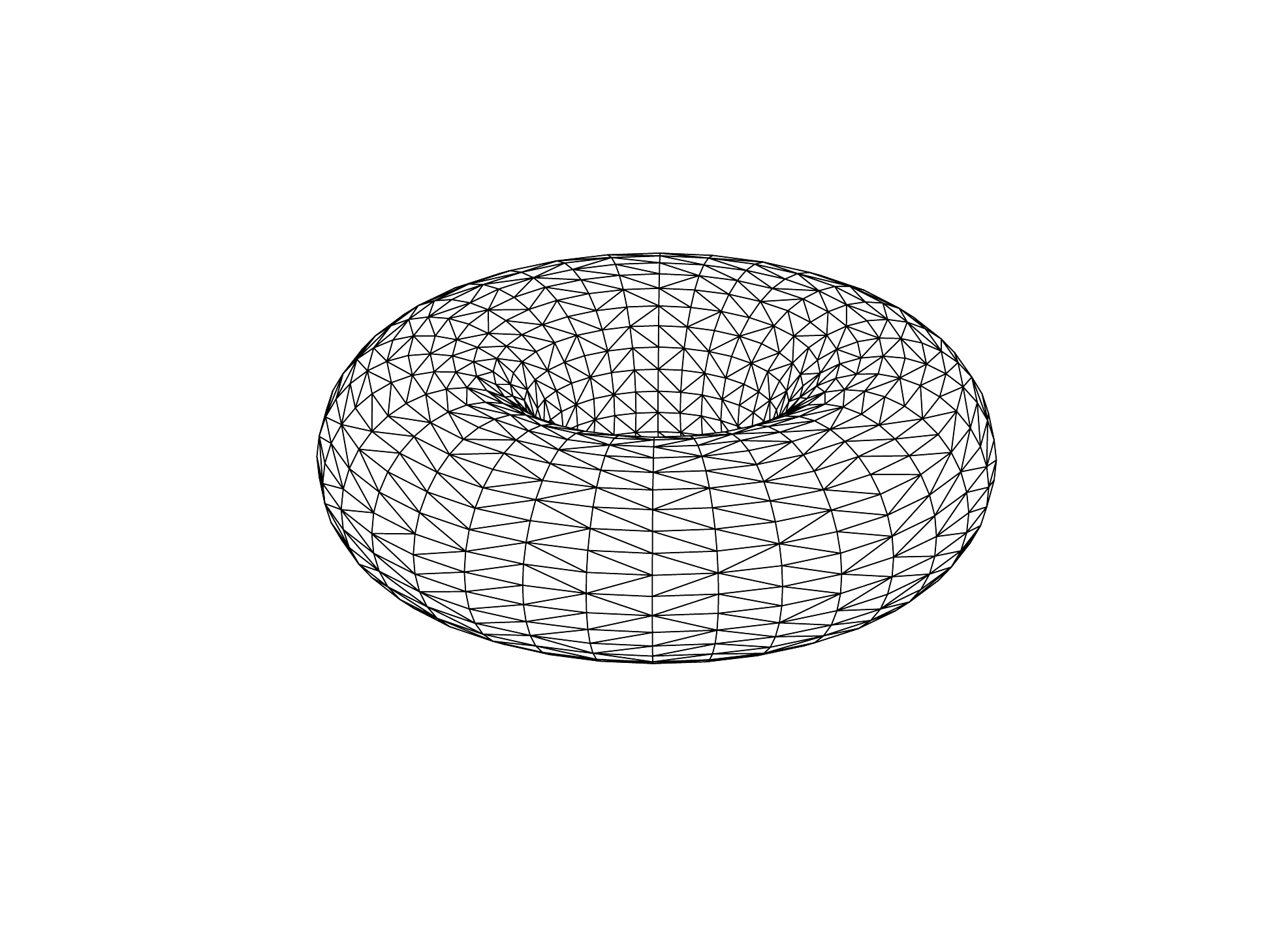}
\end{center}
\caption{Structured mesh with flipped diagonals.\label{swapped}}
\end{figure}
\begin{figure}[h]
\begin{center}
\includegraphics[height=9cm]{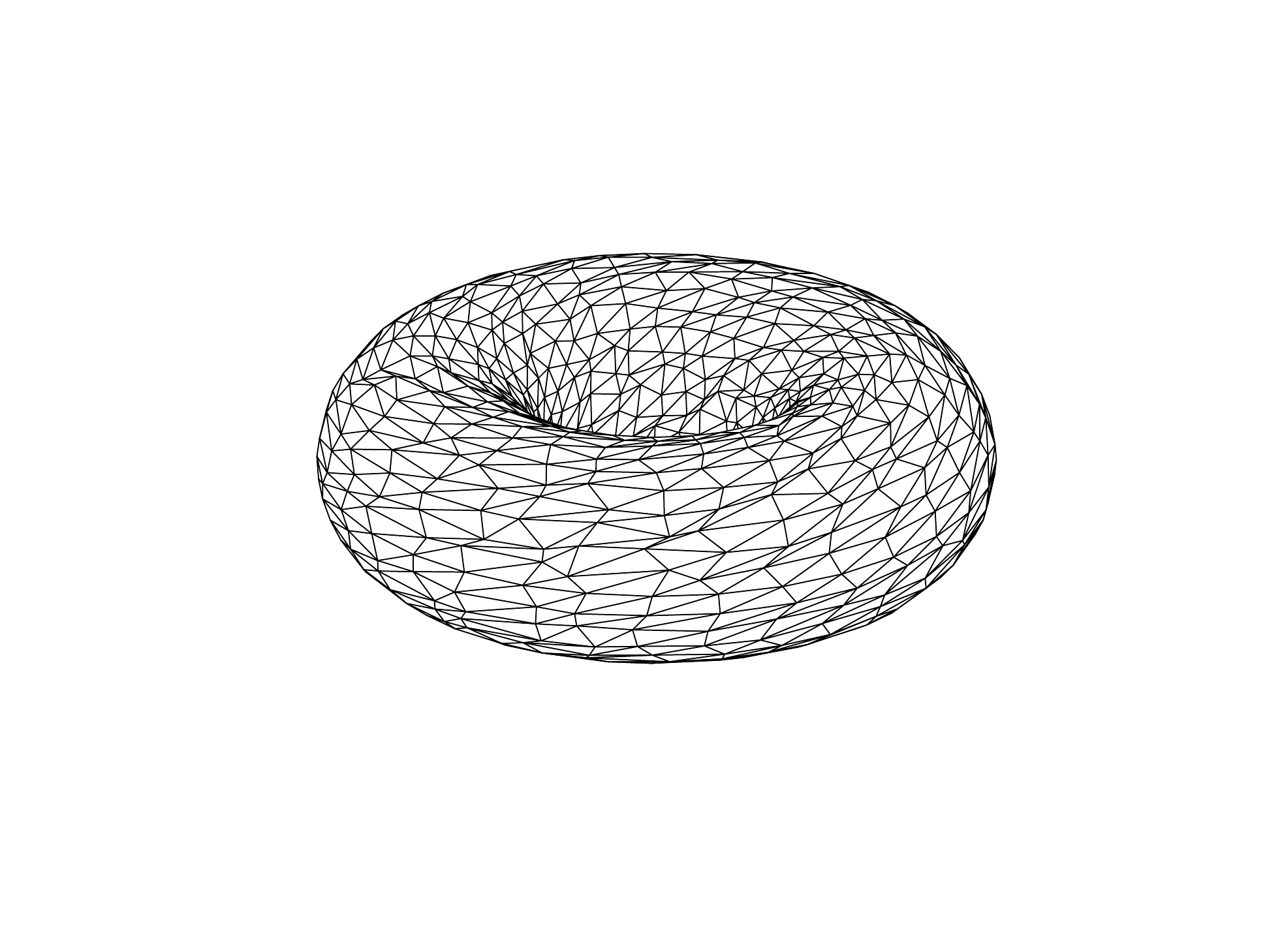}
\end{center}
\caption{Unstructured mesh.\label{unstruct}}
\end{figure}
\begin{figure}[h]
\begin{center}
\includegraphics[width=0.9\textwidth]{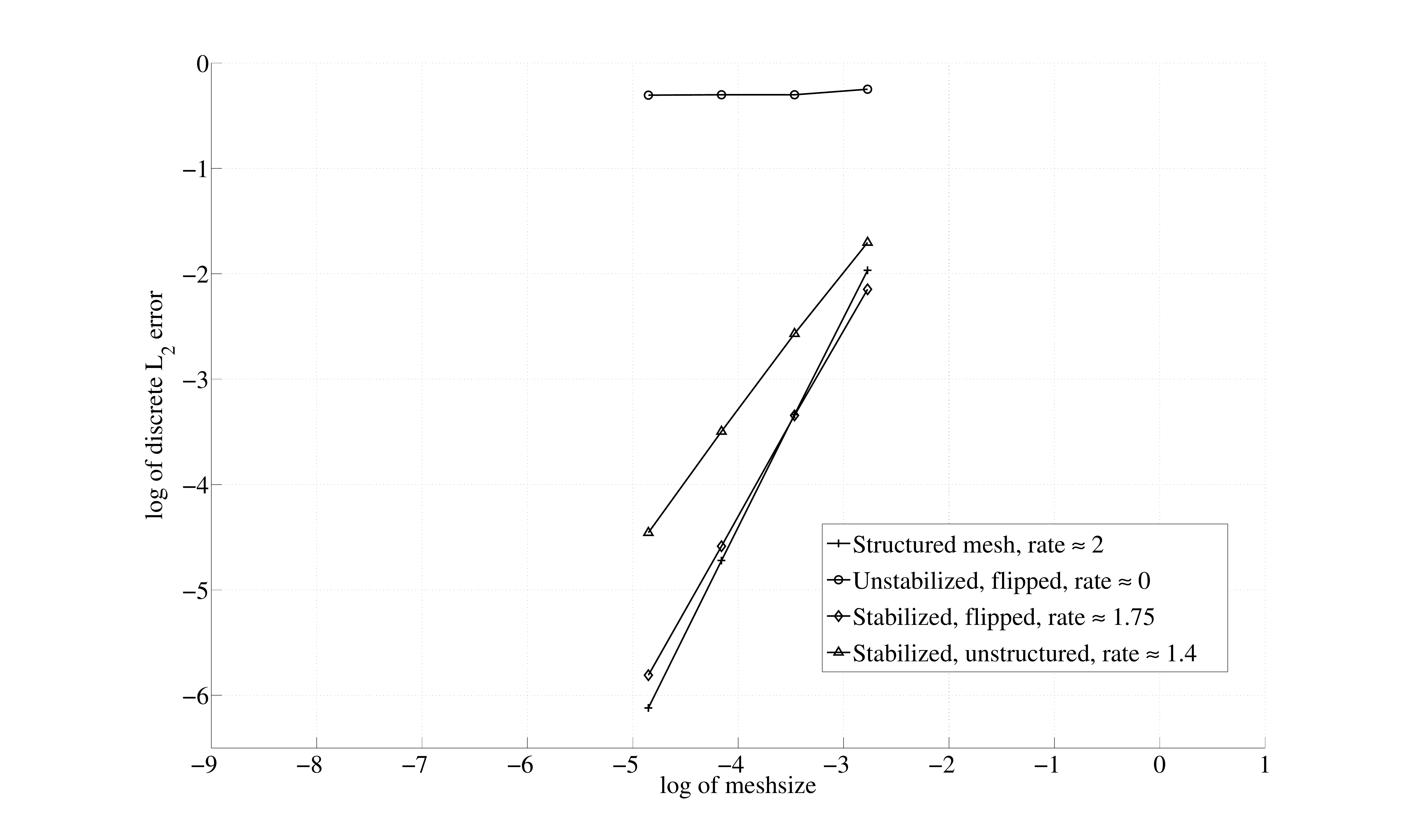}
\end{center}
\caption{Convergence curves and rates of the discrete error.\label{convtorus}}
\end{figure}
\begin{figure}[h]
\begin{center}
\includegraphics[width=0.8\textwidth]{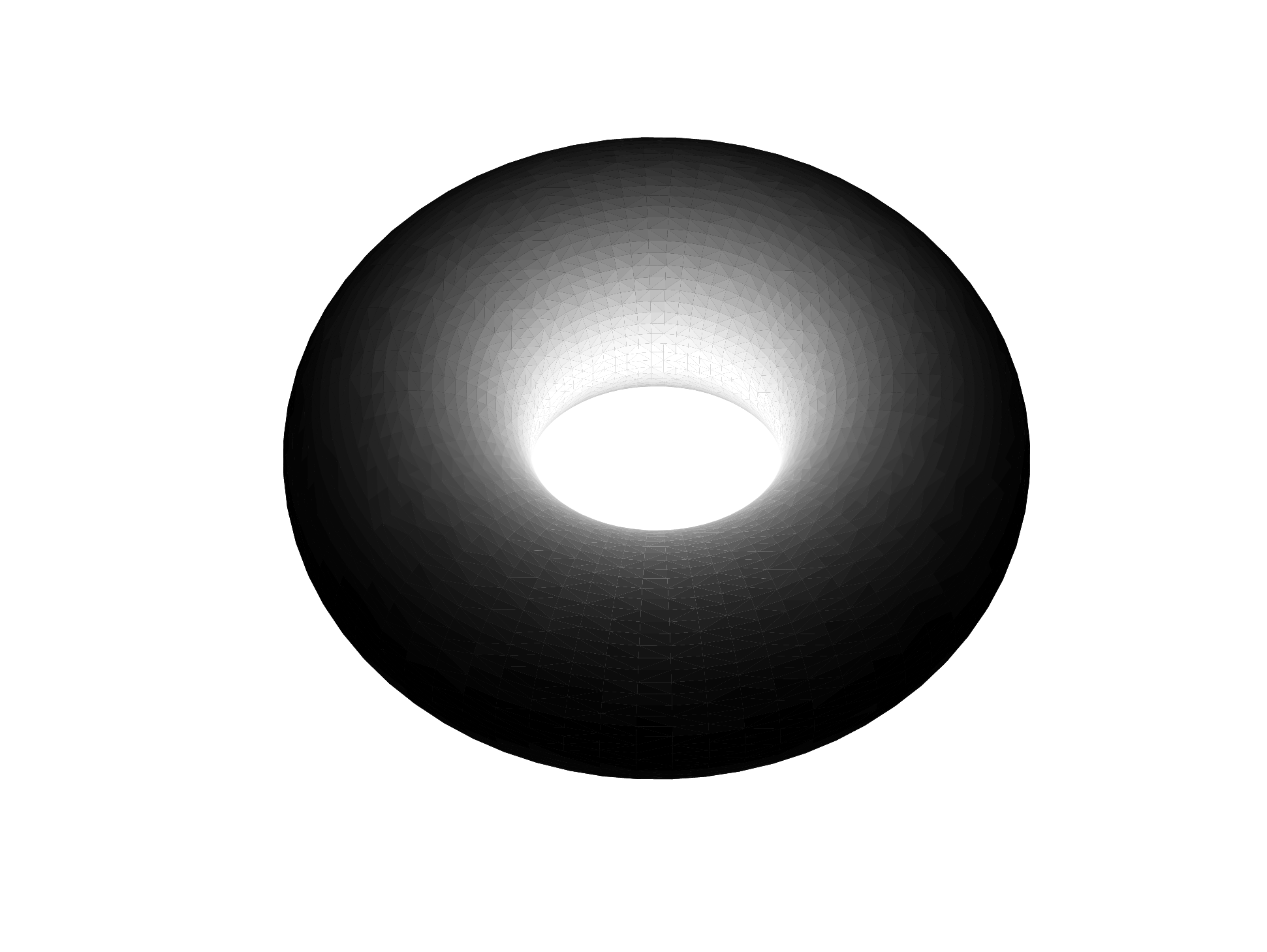}
\end{center}
\caption{Isolevels of the computed curvature, stabilized case.\label{torus1}}
\end{figure}
\begin{figure}[h]
\begin{center}
\includegraphics[width=0.8\textwidth]{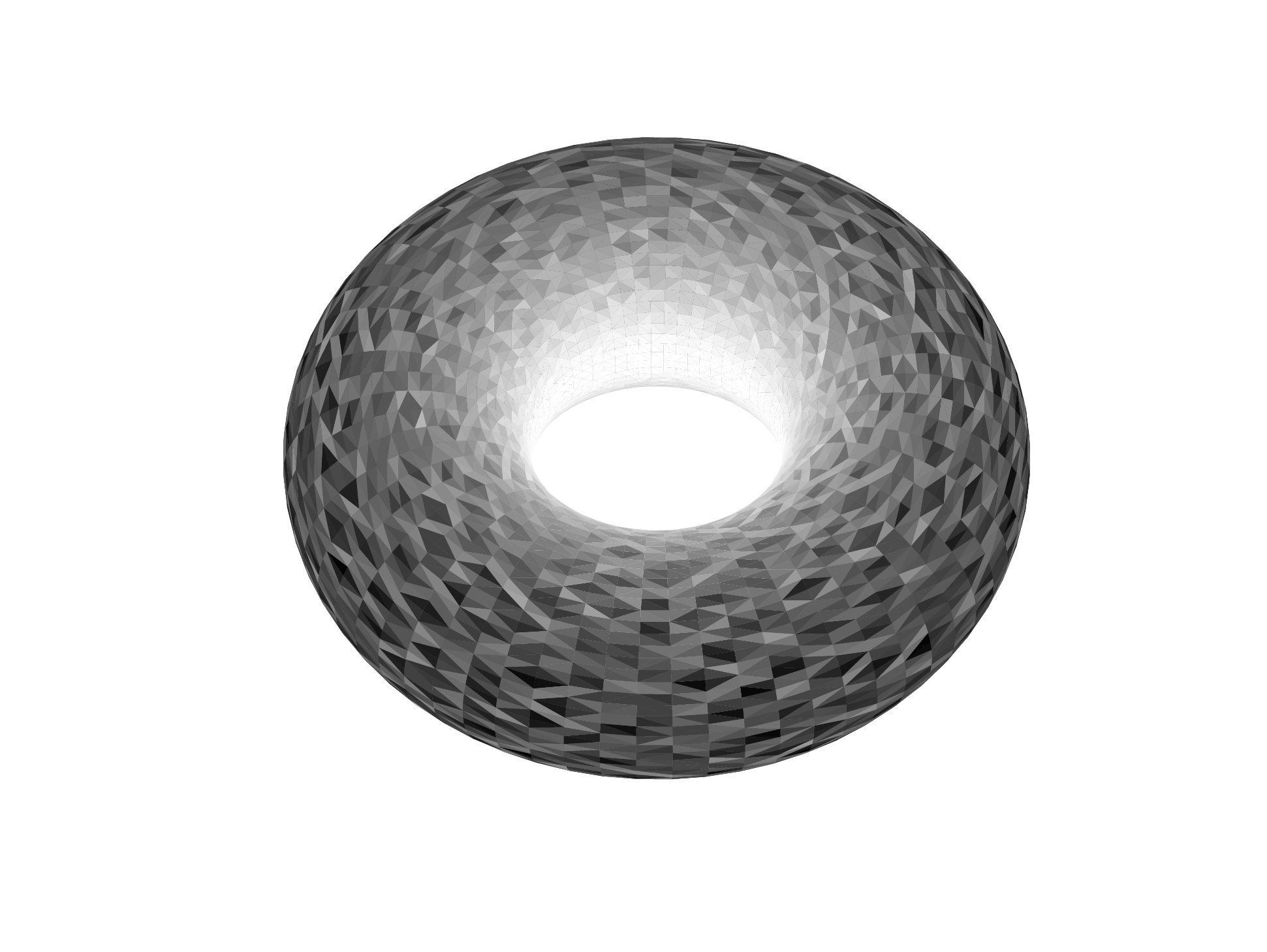}
\end{center}
\caption{Isolevels of the computed curvature, unstabilized case.\label{torus2}}
\end{figure}

\begin{figure}[h]
\begin{center}
\includegraphics[width=0.5\textwidth]{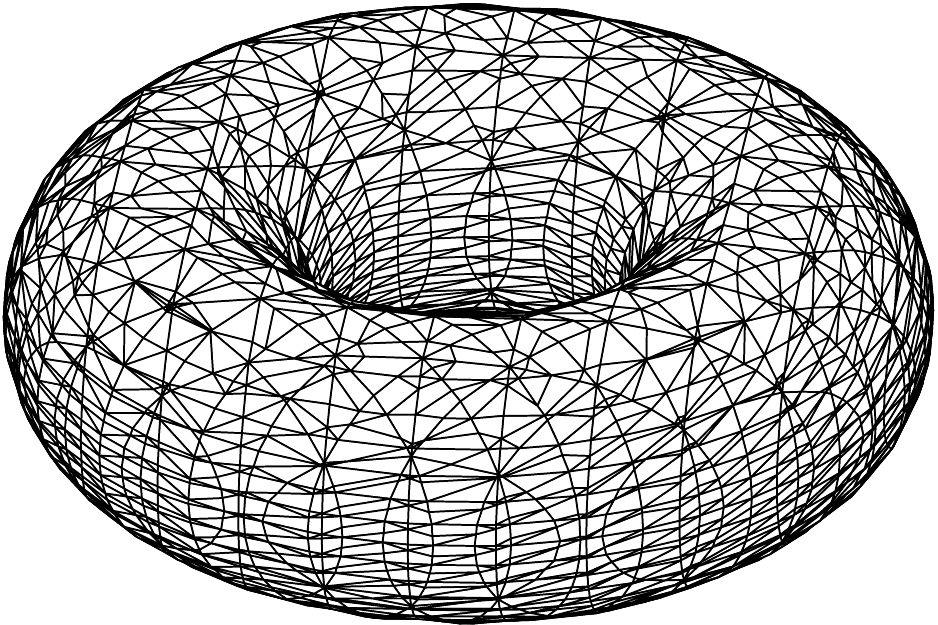}
%[height=9cm]
\end{center}
\caption{The induced triangulation of $\Sigma_h$. \label{meshtorus}}
\end{figure}

\begin{figure}[h]
\begin{center}
\includegraphics[width=0.6\textwidth]{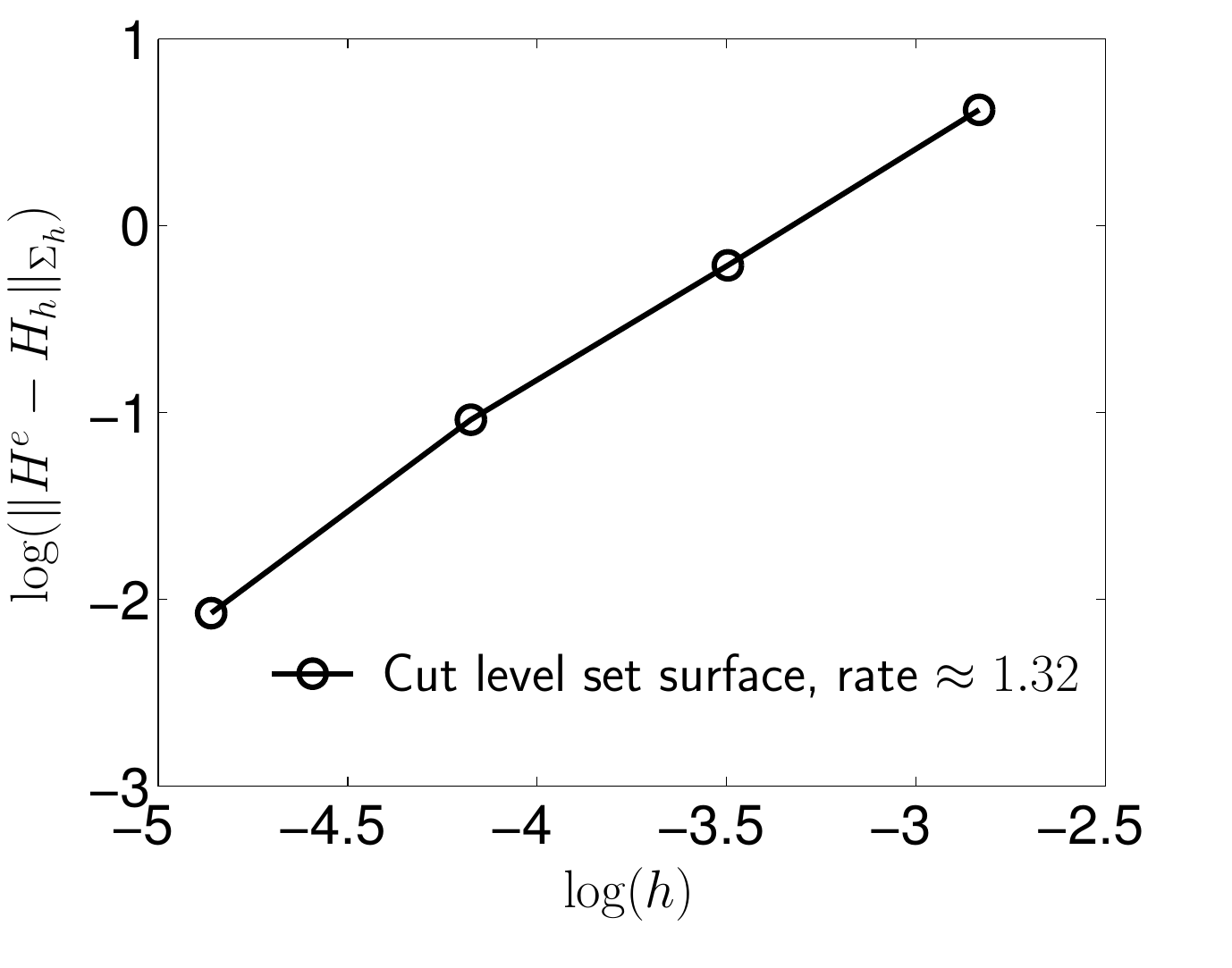}
%[height=10cm]
\end{center}
\caption{The error in the mean curvature vector for different mesh sizes. \label{convcuttorus}}
\end{figure}

%\section{Conclusions}
\end{document}